\documentclass[12pt]{amsart}
\usepackage{amscd,amssymb,latexsym}
\usepackage{fullpage}
\newcommand{\Mdef}[2]{\newcommand{#1}{\relax \ifmmode #2 \else $#2$\fi}}

\newcommand{\injdim}{\mathrm{injdim}}

\newcommand{\supp}{\mathrm{supp}}

\newcommand{\sm }{\wedge}

\newcommand{\tensor}{\otimes}

\newcommand{\sdr}{\rtimes}

\newcommand{\Hom}{\mathrm{Hom}}

\newcommand{\Ext}{\mathrm{Ext}}
\Mdef{\bhom}{\mathbf{\hat{H}om}}

\Mdef{\Mod}{\mathrm{mod}}

\newcommand{\st}{\; | \;}
\newcommand{\hash}{\#}

\newtheorem{thm}{Theorem}[section]
\newtheorem{lemma}[thm]{Lemma}
\newtheorem{prop}[thm]{Proposition}
\newtheorem{cor}[thm]{Corollary}

\theoremstyle{definition}

\newtheorem{defn}[thm]{Definition}

\newtheorem{example}[thm]{Example}

\newtheorem{remark}[thm]{Remark}

\newcommand{\qqed}{\qed \\[1ex]}
\renewenvironment{proof}[1][\hspace*{-.8ex}]{\noindent {\bf Proof #1:\;}}{\qqed}

\Mdef{\PH} {\Phi^H}
\Mdef{\PK} {\Phi^K}
\Mdef{\PL} {\Phi^L}
\Mdef{\PT} {\Phi^{\T}}

\Mdef{\ef}{E{\cF}_+}
\Mdef{\etf}{\widetilde{E}{\cF}}
\Mdef{\eg}{E{G}_+}
\Mdef{\etg}{\tilde{E}{G}}

\Mdef{\infl}{\mathrm{inf}}
\Mdef{\defl}{\mathrm{def}}
\Mdef{\res}{\mathrm{res}}
\Mdef{\ind}{\mathrm{ind}}
\Mdef{\coind}{\mathrm{coind}}

\Mdef{\univ}{\mathcal{U}}

\Mdef{\Fp}{\mathbb{F}_p}
\Mdef{\Zpinfty}{\Z /p^{\infty}}
\Mdef{\Zpadic}{\Z_p^{\wedge}}

\newcommand{\adjunction}[4]{
\diagram
#1:#2 \rrto<0.7ex> &&
#3  \llto<0.7ex> :#4 
\enddiagram}
\newcommand{\lra}{\longrightarrow}
\newcommand{\lla}{\longleftarrow}

\newcommand{\lr}[1]{\langle #1 \rangle}

\newcommand{\Gspectra}{\mbox{$G$-{\bf spectra}}}

\Mdef{\we}{\mathbf{we}}
\Mdef{\fib}{\mathbf{fib}}
\Mdef{\cof}{\mathbf{cof}}
\Mdef{\BI}{\mathcal{BI}}

\newcommand{\colim}{\mathop{  \mathop{\mathrm {lim}} \limits_\rightarrow} \nolimits}

\Mdef{\B}{\mathbb{B}}
\Mdef{\C}{\mathbb{C}}
\Mdef{\D}{\mathbb{D}}
\Mdef{\E}{\mathbb{E}}
\Mdef{\T}{\mathbb{T}}
\Mdef{\F}{\mathbb{F}}
\Mdef{\G}{\mathbb{G}}
\Mdef{\I}{\mathbb{I}}
\Mdef{\N}{\mathbb{N}}
\Mdef{\Q}{\mathbb{Q}}
\Mdef{\R}{\mathbb{R}}
\Mdef{\bbS}{\mathbb{S}}
\Mdef{\Z}{\mathbb{Z}}

\Mdef{\bA}{\mathbb{A}}
\Mdef{\bB}{\mathbb{B}}
\Mdef{\bC}{\mathbb{C}}
\Mdef{\bD}{\mathbb{D}}
\Mdef{\bE}{\mathbb{E}}
\Mdef{\bF}{\mathbb{F}}
\Mdef{\bG}{\mathbb{G}}
\Mdef{\bH}{\mathbb{H}}
\Mdef{\bI}{\mathbb{I}}
\Mdef{\bJ}{\mathbb{J}}
\Mdef{\bK}{\mathbb{K}}
\Mdef{\bL}{\mathbb{L}}
\Mdef{\bM}{\mathbb{M}}
\Mdef{\bN}{\mathbb{N}}
\Mdef{\bO}{\mathbb{O}}
\Mdef{\bP}{\mathbb{P}}
\Mdef{\bQ}{\mathbb{Q}}
\Mdef{\bR}{\mathbb{R}}
\Mdef{\bS}{\mathbb{S}}
\Mdef{\bT}{\mathbb{T}}
\Mdef{\bU}{\mathbb{U}}
\Mdef{\bV}{\mathbb{V}}
\Mdef{\bW}{\mathbb{W}}
\Mdef{\bX}{\mathbb{X}}
\Mdef{\bY}{\mathbb{Y}}
\Mdef{\bZ}{\mathbb{Z}}

\Mdef{\cA}{\mathcal{A}}
\Mdef{\cB}{\mathcal{B}}
\Mdef{\cC}{\mathcal{C}}
\Mdef{\mcD}{\mathcal{D}} 
\Mdef{\cE}{\mathcal{E}}
\Mdef{\cF}{\mathcal{F}}
\Mdef{\cG}{\mathcal{G}}
\Mdef{\mcH}{\mathcal{H}} 
\Mdef{\cI}{\mathcal{I}}
\Mdef{\cJ}{\mathcal{J}}
\Mdef{\cK}{\mathcal{K}}
\Mdef{\mcL}{\mathcal{L}}

\Mdef{\cM}{\mathcal{M}}
\Mdef{\cN}{\mathcal{N}}
\Mdef{\cO}{\mathcal{O}}
\Mdef{\cP}{\mathcal{P}}
\Mdef{\cQ}{\mathcal{Q}}
\Mdef{\mcR}{\mathcal{R}}
\Mdef{\cS}{\mathcal{S}}
\Mdef{\cT}{\mathcal{T}}
\Mdef{\cU}{\mathcal{U}}
\Mdef{\cV}{\mathcal{V}}
\Mdef{\cW}{\mathcal{W}}
\Mdef{\cX}{\mathcal{X}}
\Mdef{\cY}{\mathcal{Y}}
\Mdef{\cZ}{\mathcal{Z}}

\Mdef{\ca}{\mathcal{a}}
\Mdef{\ct}{\mathcal{t}}

\Mdef{\At}{\tilde{A}}
\Mdef{\Bt}{\tilde{B}}
\Mdef{\Ct}{\tilde{C}}
\Mdef{\Et}{\tilde{E}}
\Mdef{\Ht}{\tilde{H}}
\Mdef{\Kt}{\tilde{K}}
\Mdef{\Lt}{\tilde{L}}
\Mdef{\Mt}{\tilde{M}}
\Mdef{\Nt}{\tilde{N}}
\Mdef{\Pt}{\tilde{P}}

\Mdef{\tA}{\tilde{A}}
\Mdef{\tB}{\tilde{B}}
\Mdef{\tC}{\tilde{C}}
\Mdef{\tE}{\tilde{E}}
\Mdef{\tH}{\tilde{H}}
\Mdef{\tK}{\tilde{K}}
\Mdef{\tL}{\tilde{L}}
\Mdef{\tM}{\tilde{M}}
\Mdef{\tN}{\tilde{N}}
\Mdef{\tP}{\tilde{P}}

\Mdef{\ft}{\tilde{f}}
\Mdef{\xt}{\tilde{x}}
\Mdef{\yt}{\tilde{y}}

\Mdef{\Ab}{\overline{A}}
\Mdef{\Bb}{\overline{B}}
\Mdef{\Cb}{\overline{C}}
\Mdef{\Db}{\overline{D}}
\Mdef{\Eb}{\overline{E}}
\Mdef{\Fb}{\overline{F}}
\Mdef{\Gb}{\overline{G}}
\Mdef{\Hb}{\overline{H}}
\Mdef{\Ib}{\overline{I}}
\Mdef{\Jb}{\overline{J}}
\Mdef{\Kb}{\overline{K}}
\Mdef{\Lb}{\overline{L}}
\Mdef{\Mb}{\overline{M}}
\Mdef{\Nb}{\overline{N}}
\Mdef{\Ob}{\overline{O}}
\Mdef{\Pb}{\overline{P}}
\Mdef{\Qb}{\overline{Q}}
\Mdef{\Rb}{\overline{R}}
\Mdef{\Sb}{\overline{S}}
\Mdef{\Tb}{\overline{T}}
\Mdef{\Ub}{\overline{U}}
\Mdef{\Vb}{\overline{V}}
\Mdef{\Wb}{\overline{W}}
\Mdef{\Xb}{\overline{X}}
\Mdef{\Yb}{\overline{Y}}
\Mdef{\Zb}{\overline{Z}}

\Mdef{\db}{\overline{d}}
\Mdef{\hb}{\overline{h}}
\Mdef{\qb}{\overline{q}}
\Mdef{\rb}{\overline{r}}
\Mdef{\tb}{\overline{t}}
\Mdef{\ub}{\overline{u}}
\Mdef{\vb}{\overline{v}}

\Mdef{\hc}{\hat{c}}
\Mdef{\he}{\hat{e}}
\Mdef{\hf}{\hat{f}}
\Mdef{\hA}{\hat{A}}
\Mdef{\hH}{\hat{H}}
\Mdef{\hJ}{\hat{J}}
\Mdef{\hM}{\hat{M}}
\Mdef{\hP}{\hat{P}}
\Mdef{\hQ}{\hat{Q}}

\Mdef{\thetab}{\overline{\theta}}
\Mdef{\phib}{\overline{\phi}}

\Mdef{\uA}{\underline{A}}
\Mdef{\uB}{\underline{B}}
\Mdef{\uC}{\underline{C}}
\Mdef{\uD}{\underline{D}}

\Mdef{\bolda}{\mathbf{a}}
\Mdef{\boldb}{\mathbf{b}}
\Mdef{\bfD}{\mathbf{D}}

\Mdef{\fm}{\frak{m}}
\Mdef{\fp}{\frak{p}}

\newcommand{\fX}{\mathfrak{X}}

\Mdef{\eps}{\epsilon}

\newcommand{\cell}{\mathrm{Cell}}

\input{xypic}
\usepackage{graphicx}
\newcommand{\sub}{\mathrm{Sub}}

\newcommand{\cEi}{\cE^{-1}}

\renewcommand{\tb}{\overline{\times}}

\newcommand{\full}{\mathrm{full}}

\newcommand{\diag}{\mathrm{diag}}

\newcommand{\Vt}{\tilde{V}}

    \newcommand{\cospan}{\lrcorner}
    \newcommand{\module}{\mbox{-mod}}

\newcommand{\cOcK}{\cO_{\cK}}

\newcommand{\siftyV}[1]{S^{\infty V( #1)}}

 \newcommand{\efp}{E\cF_+}

      \newcommand{\Ntspectra}{\Nt\mbox{-{\bf spectra}}}
\newcommand{\freeGspectra}{\mbox{free-$G$-spectra}}
\newcommand{\cofreeGspectra}{\mbox{cofree-$G$-spectra}}
\newcommand{\HBGmodules}{\mbox{$H^*(BG)$-mod}}
\newcommand{\torsionHBGmodules}{\mbox{tors-$H^*(BG)$-mod}}
\newcommand{\completeHBGmodules}{\mbox{comp-$H^*(BG)$-mod}}
\newcommand{\moduleGspectra}{\mbox{-mod-$G$-spectra}}
\newcommand{\moduleNspectra}{\mbox{-mod-$N$-spectra}}
\newcommand{\modulespectra}{\mbox{-mod-spectra}}
\newcommand{\modules}{\mbox{-mod}}
\newcommand{\spectra}{\mathrm{spectra}}

\newcommand{\cSi}{\cS^{-1}}

\newcommand{\Tt}{\tilde{T}}

\newcommand{\Gp}{K^*}
\newcommand{\mcRt}{\tilde{\mcR}}
\newcommand{\cKR}{\cK \mcR}
\newcommand{\cOcKone}{\cO_{\cK_1}}
\newcommand{\cOcKzero}{\cO_{\cK_0}}
\newcommand{\cOcKR}{\cO_{\cKR}}
\newcommand{\Rep}{\mathrm{Rep}}
\newcommand{\piG}{\pi^G}

\begin{document}
\title{Algebraic models for
1-dimensional categories of rational $G$-spectra}
\author{J.P.C.Greenlees}
\address{Mathematics Institute, Zeeman Building, Coventry CV4, 7AL, UK}
\email{john.greenlees@warwick.ac.uk}

\date{}

\begin{abstract}
In this paper we give algebraic models for rational $G$-spectra for a
compact Lie group $G$ when the geometric isotropy is restricted to lie
in a 1-dimensional block of conjugacy classes. This includes all
blocks of all groups of dimension
1,  semifree spectra, and 1-dimensional blocks for many other
groups $G$.
\end{abstract}

\thanks{The author is grateful for comments, discussions  and related
  collaborations with S.Balchin, D.Barnes, T.Barthel, M.Kedziorek,
  L.Pol, J.Williamson. The work is partially supported by EPSRC Grant
  EP/W036320/1. The author  would also  like to thank the Isaac Newton
  Institute for Mathematical Sciences, Cambridge, for support and
  hospitality during the programme Equivariant homotopy theory in
  context, where later parts of  work on this paper was undertaken. This work was supported by EPSRC grant EP/Z000580/1.  } 
\maketitle
\tableofcontents

\section{Introduction}
If we pick a set $\cV$ of conjugacy classes of (closed) subgroups of  a compact
Lie group $G$ we can
consider the category $\Gspectra|\cV$ of  $G$-spectra whose
geometric isotropy is contained in $\cV$, and it is conjectured
\cite{AGconj} that
there is a small and calculable abelian model $\cA (G|\cV)$ so that the category DG-$\cA
(G|\cV)$ of differential graded objects in $\cA (G|\cV)$ is a model
for $\Gspectra|\cV$ in the sense that there is a Quillen equivalence
$$\Gspectra|\cV\simeq DG-\cA (G|\cV).$$
The purpose of the present paper is to verify the conjecture in a
range of particularly simple cases. 

The space $\fX_G=\sub(G)/G$ of conjugacy classes of closed subgroups
of $G$ has two topologies of interest.
First of all, it has the h-topology, which is the quotient topology of the
Hausdorff metric topology on $\sub (G)$. It also has the Zariski
topology, whose closed sets are the h-closed sets which are also
closed under passage to cotoral subgroups\footnote{$L$ is {\em cotoral} in $K$
if $L$ is normal in $K$ with quotient being a torus. For conjugacy
classes $(L)$ is cotoral in $(K)$ if the relation holds for some
representative subgroups.}. Since $\cV$  is a subset of $\fX_G$ it
also acquires two topologies, though it will not necessarily be a
spectral space. 

The very simplest case is when $\cV$  is h-discrete and there are no non-trivial
cotoral relations (we say $\cV$ is of Thomason height 0).  In this
case the conjecture is immediate from results of 
\cite{gfreeq2}, and the model is described by 
$$\cA(G|\cV)\simeq \prod_{H\in \cV}\cA (G|H)$$
and 
$$\cA (G|H)\simeq \mbox{tors-}H^*(BW_G^e(H))[W_G^d(H)]\modules,  $$
where $W_G(H)=N_G(H)/H$ has identity component $W_G^e(H)$ and
component group $W_G^d(H)$.

It is the purpose of the present paper to prove the conjecture in a
range of  cases when $\cV$ is one step more complicated, so that it has a subspace
$\cK$ of Thomason height 0 and one additional point $\Gp$. The symbol
$\Gp$ will be used for this point throughout this paper. Thus we either have a cotoral
inclusion $K\leq_{ct}\Gp$ for some $K\in \cK$ or else $\Gp$ is an h-limit of points of
$\cK$.  We will also restrict the group theory by assuming $\cK$ consists of finite groups, and 
$\Gp$  has finite index in its normalizer (most examples we make 
explicit  also have $\Gp=G$). This is restrictive, but it permits a simple
exposition and covers the cases of immediate interest.
Section \ref{sec:generalizations} describes some of the complications
that arise as more cases are covered.

The method is closely based on that used for the circle group
and  for the two blocks of  $O(2)$, but it is more uniform and
covers many other cases. The cases covered  include all blocks of
1-dimensional groups and semi-free $G$-spectra when $G$ is a torus.
It also covers the block of full subgroups of a toral group (the
identity component $G_e=\T$ is a torus)  where the $G_d$-module
$H_1(T;\Q)$ is simple. Already these  cases display a range of different behaviour.

It is helpful to bear some examples in mind.

\begin{example}
  \label{eg:intro}
\begin{enumerate}
\item The space $\cV$ need not be
  irreducible. This means that
  $\Gp$ is not necessarily a generic point.

  For example if $\cV$ in the h-topology is the one point compactification of a
countable set and the cotoral relation is trivial. This is the case for the
dihedral subgroups of $O(2)$.

\item The space need not be compact in the h-topology. For example we may add infinitely
  many points to a compact example  without adding the limit point. It is easy to construct such examples 
  with $G$ the 2-torus (see (4)). 

  \item A more extreme form of this is when every neighbourhood of
    $\Gp$ contains an infinite set of which it is not the limit. For
    example $G=\Gp=T^2$ is the 2-torus, and the other points are the
    finite subgroups of $C_1\times T, C_2\times T, C_3\times T, \ldots$. 
\item  The space $\cV$ may be discrete in the  h-topology. For example we
  may take $G=\Gp=T^2$ to be the 2-torus and  $\cK$ to consist of the
  proper subgroups of $T\times 1\subseteq T\times T=G$.
\end{enumerate}
\end{example}

  It is natural to want to generalize these results. 
The arguments presented here also prove the conjecture in many other 1-dimensional 
cases, and the natural level of generality is not yet
clear; some cases may best be treated as subcategories of higher
dimensional examples. In any case,  our focus
here will be on expository simplicity rather than greatest
generality. Section \ref{sec:generalizations} discusses ways to relax
the restrictions.

\subsection{Associated work in preparation}
This paper is the second in a series  of 5 constructing an algebraic
category $\cA (SU(3))$ and showing it gives an algebraic model for 
rational $SU(3)$-spectra. This series gives a concrete 
illustrations of general results in small and accessible 
examples.

The first paper \cite{t2wqalg} describes the group theoretic data that feeds into the construction of an
abelian category $\cA (G)$ for all toral groups $G$ and makes them
explicit for toral subgroups of rank 2 connected groups.
The present paper (which does not logically depend on \cite{t2wqalg})
constructs algebraic models for all relevant 1-dimensional
blocks. The paper \cite{t2wq} gives an
algebraic model for the maximal torus normalizer
in $U(2)$.

The paper \cite{u2q} assembles this information and that
from \cite{gtoralq} to give an abelian category $\cA (U(2))$ in 7
blocks and shows it is an algebraic model for rational
$U(2)$-spectra. Finally, the paper \cite{su3q} constructs $\cA (SU(3))$
in 18 blocks and shows it is equivalent to the category of rational
$SU(3)$-spectra. The most complicated parts of the model for $G=U(2)$
and $G=SU(3)$ are the toral blocks, which are based on the work in the
present paper.

This series is part of a more general programme. Future installments
will consider blocks with Noetherian Balmer spectra \cite{AGnoeth} and
those with no cotoral inclusions \cite{gqwf}. 
An account of the general nature of the models is in preparation
\cite{AVmodel}, and the author hopes that this will be the basis of the proof that the
category of rational $G$-spectra has an algebraic model in general.

\subsection{Contents}

The paper is divided into two parts. Part 1 is essentially
algebraic and Part 2 shows that appropriate categories of rational
$G$-spectra are equivalent to the algebraic models.

The main part of the paper gives an algebraic model for a
1-dimensional block with a single dominant subgroup. To show that this
is useful, we start in Section \ref{sec:subgroups} by showing that for
a 1-dimensional group, the space of subgroups can be decomposed into
such blocks, and furthermore the auxiliary data can be describe in
these terms. Thereafter we have in mind a single 1-dimensional block.
In Section \ref{sec:ingredients} we describe the data required to build the
model, and in Section \ref{sec:sheafdatatop} we describe how this data can be
obtained for compact Lie groups.
In Section \ref{sec:abelian} we explain how to build the model from
the data, and we show how to work with the category, inparticular
proving it  is of injective dimension 1. 

We begin Part 2 in Section \ref{sec:genstrat} by describing the general
strategy for showing a 1-dimensional block has an algebraic model. 

It is rather well known that for a connected compact Lie group $G$,
the category of cofree $G$-spectra is a category of complete modules
over $H^*(BG)$, and that this  can easily be adapted to disconnected groups.
In Section \ref{sec:DEGp} we show that the category of modules over
$DEG_+$ is the corresponding category of modules over
$H^*(BG)$ without any completeness condition, and that the model for cofree $G$-spectra is obtained by
completion. This gives the stalkwise model, and it is quite easy to
assemble this algebraically over the height 0 subgroups. However for
the splicing to the height 1 subgroup we need to have a common
framework. The new complication here is having to deal with the
varying component structure, and the
method for doing that is described in Section \ref{sec:unif}. This
completes the account of the model over height 0 subgroups, and 
finally in Section \ref{sec:dimonemodels} we assemble the information
over height 0 and height 1 groups  to establish that 1-dimensional blocks are algebraic. 
In Section \ref{sec:generalizations} we review the technical
difficulties that need to be tackled to give further generalization.

\part{Algebra}
\section{Full subgroups}
\label{sec:subgroups}

The purpose of this section is to show that for 1-dimensional groups $G$
the category of $G$-spectra breaks into pieces each one of which is
covered by our general analysis. This comes in two steps. Firstly, we
establish a decomposition  indexed by conjugacy classes of
subgroups $V$ of the component group $W$. Secondly, we show that
it suffices to consider the case when $V=W$, in the sense that the
model in the general case can often be easily deduced from the full case. 

For the purpose of this section, 
we suppose that $G$ has identity component $G_e=\T$ a torus and lives 
in an extension 
$$1\lra \T \lra G\stackrel{\pi}\lra W\lra 1.$$

\subsection{Partition by subgroups of the component group}
For any compact Lie group $G$, one may show that the space $\sub(G)/G$
can be partitioned into blocks dominated by a single group. When $G$
is a toral group there is a partition that takes a particularly simple
form. Since any 1-dimensional group $G$ has identity component a circle,
this case is covered automatically.

\begin{lemma}
  \label{lem:decomp}\cite[2.1]{t2wqalg}
For a toral group $G$ as above, the space 
$\fX_G=\sub(G)/G$ of conjugacy classes of subgroups of the toral group 
$G$ is partitioned into pieces, $\cV^G_{\Hb}$  one for each conjugacy
class of subgroups $\Hb$ of $W$. 

If $\Hb\subseteq W$, the set 
$$\cV_{\Hb}^G=\{ (K)\st \pi (K)=\Hb\}$$
is clopen in the Hausdorff metric topology and 
closed under passage to cotoral subgroups. Furthermore, 
$\cV^G_{\Hb}$ is dominated by $\pi^{-1}(\Hb)$ in the sense that it
consists of all subgroups cotoral in $\pi^{-1}(\Hb)$.

Accordingly, the  Balmer spectrum with its Zariski topology is a coproduct 
$$\fX_G=\coprod_{(\Hb)}\cV_{\Hb}^G. $$
\end{lemma}

\begin{remark}
(a) The partition is crude, in the sense that the 
 sets $\cV^G_{\Hb}$ can often be decomposed further. 

  (b) Subgroups mapping onto $W$ are called {\em full}, and the
 component of full subgroups is $\cV^G_W$. Most of the rest of the
 paper will focus on full subgroups because the component of
 $\cV^G_V$ can be studied in the group $\Vt:=\pi^{-1}V$ as explained
 in Subsection \ref{subsec:fullreduction} below.

 (c) Since subgroups of $G$ with image $V$ are by definition subgroups
 of $\Vt$, the map $\cV^{\Vt}_V\lra \cV^G_V$ is surjective, and the
 only effect is fusion of $\Vt$-conjugacy classes to form
 $G$-conjugacy classes.  Fusion can nonetheless have significant
 effects (for example the map $\fX_{\T}\lra \fX_G$ factors
 through $\fX_{\T}\lra \fX_{\T}/W$). 
 \end{remark}

 \subsection{Reduction to full subgroups}
 \label{subsec:fullreduction}
Continuing with a toral group $G$,  in analysing $\cV^G_V$,
we may reduce to the case that $V$ is normal in $W$.

To see this, let $N=N_W(V)$. The $W$-conjugacy class $(V)_W$ may split
into several $N$-conjugacy classes $(V)_W=\coprod_i (V_i)_N$. There is
a corresponding splitting of the idempotent $e_V$:
$\res^W_N(e_V)=\sum_i e_{V_i}$. These idempotents may be inflated to
$G$, and we see
$$\Ntspectra|\cV^G_V\simeq e_V\Ntspectra \simeq \prod_i
e_{V_i}\Ntspectra\simeq \prod_i \Ntspectra|\cV^{\Nt}_{V_i}. $$
Since the $V_i$ are conjugate in $G$, the factors are equivalent using
conjugation by elements of $G$. For subgroups $H$ with $\pi
H=V$, the relevant map of underlying spaces is
$$\coprod_i \cV^{\Nt}_{V_i}\lra \cV^G_V.$$

\begin{lemma}
  \label{lem:reducetonormal}
  Suppose $V\subseteq W$ and let $N=N_W(V)$.
  Restriction
induces  a full and faithful functor
$$\Gspectra|\cV^G_V\lra \Ntspectra|\cV^{\Nt}_V. $$
 The essential image 
consists of $\Nt$-spectra which are constant on the factors
$\Ntspectra|\cV^{\Nt}_{V_i}$ in the sense that they correspond under
conjugation by $G$. Composing with restriction to one factor, we
obtain  an equivalence 
$$\Gspectra|\cV^G_V\simeq  \Ntspectra|\cV^{\Nt}_V. $$
\end{lemma}

\begin{remark}
  This shows directly that the two categories should have equivalent
  models. In fact we observe that  $N_G(V)=N_{\Nt}(V)$, so that the
  sheaf of rings and component structures also agree. Thus,
  corresponding to the equivalence of categories of $G$-spectra we
  have an equivalence
   $$\cA (G|\cV^G_V)=\cA (\Nt|\cV^{\Nt}_V).$$
  \end{remark}

\begin{proof}
  We must show the map is full, faithful and essentially surjective.
  
  The map
  $$[X,Y]^G\lra [G_+\sm_N eX, Y]^G=[eX,Y]^N$$
is induced by   $G_+\sm_{\Nt} S^0\sm X \simeq G_+\sm_{\Nt} eX \lra X$ of
$G$-spectra. Now $G/\Nt=W/N$, and the $V$ fixed points consist of
cosets $wN$ so that  $V^w\subseteq N$. The fact that restriction is an
isomorphism follows since the map $W_+\sm_V  eS^0\lra S^0$ is an
equivalence in $V$-fixed points.
\end{proof}

Next, we may further reduce to working with $\cV^{\Vt}_V$ if we take into
account the action of the Weyl group $N_W(V)/V$.

\begin{lemma}
 If  $N=N_W(V)$, the geometric isotropy of $X$ consists of
subgroups of $\Vt$ then restriction induces an isomorphism 
$$[X,Y]^{\Nt} \stackrel{\cong}\lra \left( [X,Y]^{\Vt}\right)^{N/V}. $$
\end{lemma}

\begin{proof}
We have an equivalence $EN/V_+\sm X\simeq X$, and the spectral
sequence of the skeletal filtration of $EN/V_+$ gives the
isomorphism. 
\end{proof}

\begin{remark}
This shows that to understand $\Nt$-spectra we need only understand
$\Vt$ spectra together with an action of the finite group
$N/V$. However one does need to bear in mind that $N/V$ acts on the
space of subgroups as well as all other elements of the construction. 

Combining this with Lemma \ref{lem:reducetonormal}, as in \cite[6.10]{AGtoral}, this
shows we have a reduction to the case of full subgroups, and we may take
$$\cA (G|\cV^G_V)\simeq \cA (\Nt|\cV^{\Nt}_V)\simeq ``\cA
(\Vt|\cV^{\Vt}_V)[N/V]''.  $$
Note that in this statement $N/V$ may act non-trivially on
$\cV^{\Vt}_V$, so the right hand category is not simply
$\cA (\Vt|\cV^{\Vt}_V)$ with an action of $N/V$, and considerable
elucidation along the lines of \cite{AGtoral}  is necessary.  With
these caveats, we may reduce to the case of full subgroups.

This strategy is especially effective if $G$ is 1-dimensional since the action of $N/V$
on the identity component does not permute subgroups. In this case we
find
$$\cV^{\Nt}_V=\cV^{\Vt}_V, $$
and in fact the Weyl groups are of the same dimension so that the structure
sheaves of rings will agree. The component groups will differ since
the group $N/V$ certainly acts for $\Nt$-spectra. Indeed, 
 if $\pi (K)=V$ there is a map 
$$N_{\Nt}(K)\lra N_N(V)=N, $$
since $K\cap T$ is characteristic, and hence there is a map
$$W_{\Nt}(K)\lra N/V, $$
but this need not be an isomorphism for all $K$.
\end{remark}

\section{Ingredients for the abelian models}
\label{sec:ingredients}

The rest of our analysis supposes that we are considering $G$-spectra
with geometric isotropy in a 1-dimensional space $\cV$ of a special
form. In fact we suppose given a countable set $\cK$ of conjugacy
classes subgroups. As a subspace of $\fX_G$ it is discrete
and we assume $\cV=\cK\amalg \{\Gp\}$.

For example if $G$ is a 1-dimensional group we may choose a subgroup
$V$ of the component group and suppose $\cK$ consists of the finite
subgroups $H$ of $G$ with $\pi (H)=V$ and $\Gp=\pi^{-1}(V)$.

We discuss the form of the space $\cV$ in Subsection
\ref{subsec:formofcV}, and the auxiliary data in Subsection
\ref{subsec:auxiliary}. 

\subsection{Decomposing $\cV$}
\label{subsec:formofcV}

We start with a countable set $\cK$ of conjugacy classes of 
finite subgroups, and $\cV=\cK\amalg \{\Gp\}$. It is convenient to
partition $\cK$. 

\begin{lemma}
  There is  a partition $\cK=\cK_{0}\amalg \cK_1 \amalg \cKR$ into
  Zariski clopen sets where
  \begin{itemize}
  \item $\cK_1$ consists of subgroups cotoral in $\Gp$.
    \item  $\cK_0$ consists of subgroups not cotoral in $\Gp$ but
      $\cK_0$ has $\Gp$ as a limit point and
  \item the remainder $\cKR$ has no limit points in $\cV$.
\end{itemize}
\end{lemma}
\begin{proof}
  If $\Gp$ is not a limit point when we remove $\cK_1$ we take
  $\cKR=\cK\setminus \cK_1$ and $\cK_0=\emptyset$. Otherwise, we
  choose an h-neighbourhood $U$ of $\Gp$, and take  $\cK_0':=
(\cK\setminus \cK_1)\cap U$ and  $\cKR:=(\cK \setminus 
\cK_1)\setminus U$.
\end{proof}

\begin{remark} It is disappointing that we cannot generally give a
  canonical decomposition, and that the pieces cannot be simpler. 

  (i) Example \ref{eg:intro} (3) shows that it may happen that every
  choice of $\cK_0$ contains an infinite set $\cKR'$ without $\Gp$ as a limit
  point, so $\cKR'$ could be moved into $\cKR$.

(ii) Even if $\Gp$ is a limit point of every infinite subset of
$\cK_0$,  the partition is not canonical, since any finite set of points 
can be moved between $\cK_0$ and $\cKR$.
\end{remark}

\begin{remark}
(i) If $\cKR=\emptyset $ and  $\Gp$ is a limit point of every
infinite subset of $\cK_0$, we say $\cV$ is 
{\em   almost irreducible}.
  
(ii)   The topology is
  determined by the closed subsets of $\cV$ not containing $\Gp$. This
  contains the finite subsets of $\cK$, and if it contains no other
  subsets, $\cK \cup \{ \Gp\}$ is the one point compactification of $\cK$. 
  \end{remark}

We
discuss some examples and then show that in many cases it suffices to
deal with the special cases $\cK=\cK_0$ (Type 0) and $\cK=\cK_1$ (Type 1).

The models 
differ in character according to the sizes of $\cK_0$ and $\cK_1$: even in this very special 
context, a wide variety of behaviours is possible.

\begin{example}
  \label{eg:extras}
Here is a small selection of almost irreducible  examples with $\cK_1$ infinite. 
\begin{enumerate}
\item  If $G=T^2$ all subgroups are cotoral in $G$, so in all cases
$\cK=\cK_1$. Nonetheless these cases can vary in character.
\begin{enumerate}
\item If $\cK$ consists of all subgroups $T^2[n]$ then $G=\Gp$ is the  1-point compactification.
\item If $\cK=\cK_1 =\{ T[n]\times T \st n \geq 1\}$
then the h-topology is discrete.
\item If $\cK=\cK_1 =\{ T^2[n] \st n \geq 1\} \cup \{ C_m\times 1\st m\geq 
1\}$ then there is an infinite subset of $\cK_1$ without limit
points. 
\item If $\cK$ consists of all finite subgroups, then it has infinitely
many limit points. 
\end{enumerate}

\item  We may have both $\cK_0$ and $\cK_1$ infinite. For example 
 $G=T\times O(2)$ and then take $\Gp=G$ with 
 $\cK_1=\{ C_m\times O(2)\st m\geq 1\}$, and $\cK_0= \{ T\times D_{2n}
 \st n \geq 1\}$. 
\end{enumerate}
\end{example}

A property holds  {\em almost everywhere} in $\cK$ if it
holds except for a closed subset of $\cV$ lying in $\cK_0\amalg \cKR$
(in the almost irreducible case, this means for all but a finite
number of points of $\cK_0$). 

Decompositions of $\cV$ lead to decompositions of the model. We
illustrate with two special cases.

\begin{example}
  (i) We may always write $\cV=\cV'\amalg \cKR$ and then
  $$\cA (G|\cV)=\cA (G|\cV')\times \cA (G|\cKR), $$
  where $\cA (G|\cKR)$ is described in the introduction. The existence of
  idempotents in the Burnside ring $A(G)$ shows there is a corresponding decomposition of
  spectra, so for most purposes we may assume $\cKR=\emptyset$.

  (ii) If $\cV$ is almost irreducible then 
  $\cV=\cV_0\vee \cV_1$ (with the wedge point being $\Gp$)
  where $\cV_0=\cK_0^\hash$ (one point compactification) and
  $\cV_1=\cK_1\cup \{\Gp\}$. There is then  a pullback
square
$$\xymatrix{
  \cA (G|\cV)\rto \dto &\cA(G|\cV_0)\dto\\
    \cA (G|\cV_1)\rto  &\cA(G|\Gp)
  }$$
  and similarly for spectra.

  An important example example consists of subroups of $G=N_{U(2)}(T^2)$
  of dimension $\geq 1$ with $\Gp=G$, where $\cK_0$ consists of 1-dimensional
  subgroups containing the central circle, and $\cK_1$ consists of
  1-dimensional subgroups containing all elements $\diag(\lambda,
  \lambda^{-1})$. 
\end{example}

  \subsection{Auxiliary data}
  \label{subsec:auxiliary}
The model for $G$-spectra with geometric isotropy in $\cV$ is a
category whose objects are equivariant sheaves of modules over a
sheaf of rings over $\cV$. In the 1-dimensional case, rather than
making explicit all of the adjectives to make this precise, we will give the
data required directly.  We require  the additional
data of a `sheaf of rings' $\mcR$ and a `component structure' $\cW$ which specifies
the equivariance together with a `coordinate structure' $\cS$ linking
the topology and algebra. 

\begin{defn}
(a) A {\em sheaf of rings} $\mcR$ on $\cV$ consists of commutative
polynomial rings $\mcR(F)$ for
$F\in \cK$ and $\mcR(\Gp)$, and
a ring homomorphisms $\mcR(\Gp)\lra \mcR(F)$ for almost all $F$.

(b) A {\em component
structure} on $\cV$  consists of finite groups $\cW_F$ for $F\in \cK$ and $\cW_{\Gp}$,
together with a homomorphism $\cW_F\lra \cW_{\Gp}$ for almost all $F$.

(c) The  component
structure $\cW$ has an action on $\mcR$ if $\cW_F$ acts on $\mcR(F)$
and $\cW_G$ acts on $\mcR(\Gp)$ so that $\mcR(G)\lra \mcR(F)$ is
$\cW_F$-equivariant. 
\end{defn}

We also need to relate the topology on $\cV$ to the algebra of the
sheaves. For this we need functions whose vanishing determines the topology.

\begin{defn}
  An {\em coordinate structure} on a sheaf of rings om $\cV$ is a collection of
  multiplicatively closed sets $\cS_{\Gp/F}$ in $\mcR(F)$ for $F\in
  \cK_1$. For $F\in
  \cK_0$, we take $\cS_{\Gp/F}=\{ 0, 1\}$, and for $F\in \cKR$ we take
  $\cS_{\Gp/F}=\{0\}$.

If there is a compatible component structure we say $\cS_{\Gp/F}$ is
compatible if the elements of $\cS_{\Gp/F}$ are invariant under the action. 
\end{defn}

We will show in Section \ref{sec:sheafdatatop} below that there is a sheaf of
rings and a compatible component structure arising in the  Lie group
context. 

\section{Equivariant sheaf data in the geometric context}
\label{sec:sheafdatatop}
In this section we explain describe the necessary auxiliary data
$(\mcR, \cW, \cS)$ from  Subsection \ref{subsec:auxiliary} in the motivating example of a
compact Lie group. It is notable that we need to use the fact we are
working over the rationals to give a splitting principle to obtain the
structure in full generality.

\subsection{Sheaves and component structures}
The simplest way to ensure a sheaf of rings with compatible component
structure is to have a       homomorphism $N_G(K)\lra
      N_G(\Gp)$ extending an inclusion $K\subseteq \Gp$.  
In this case it is straightforward since there is an      induced map
      $$W_G(K)=N_G(K)/K\lra N_G(\Gp)/\Gp=W_G(\Gp). $$
      The map on identity components supplies a map $\cO_K\lla \cO_{\Gp}$,
      and it is equivariant for the map $W_K=W_G^d(K)\lra
      W_G^e(\Gp)=W_{\Gp}$.

      We may ensure this is the case for $K\in \cK_0$ since the
      normalizer construction is upper semi-continuous \cite[9.8]{prismatic} so
      that $N_G(K)\subseteq N_G(\Gp)$ for almost all $K\in \cK$. 
It is often reasonable for cotoral inclusions too: for example we might suppose
$G=N_G(\Gp)$ for all $H\in \cV_1$. Often we may be able to reduce to
this case. For example if $\cV_1$ is a singleton $(\Gp)$ we have the 
  familiar question of how to allow for the restriction from $G$ to 
  $N_G(\Gp)$.

  However we want to cover the general case (for example if  $1\in
  \cK$). For this we proceed as   follows.   After conjugation, we may
  suppose $K\subseteq \Gp$ with $\Gp/K$ is a torus. We may then choose a maximal torus of
  $W_G(K)$ containing $\Gp/K$ and  write $\Tt$ for its inverse image in $G$. 
 We see (since the cotoral relation is transitive \cite{ratmack})
 that both $\Gp$ and $K$ are cotoral $\Tt$. The image $\Tt/\Gp$ of $\Tt$
 in $W_G(\Gp)$ is a torus, so we can choose the maximal torus of
 $W_G(\Gp)$ to contain it. In fact $\Tt/\Gp$ is itself a
 maximal torus, since otherwise its inverse image $\Tt_1$ would have
 image $\Tt_1/K$ in $W_G(K)$ properly containing the maximal torus
 $TW_G(K)$. The situation is depicted below.  

 $$
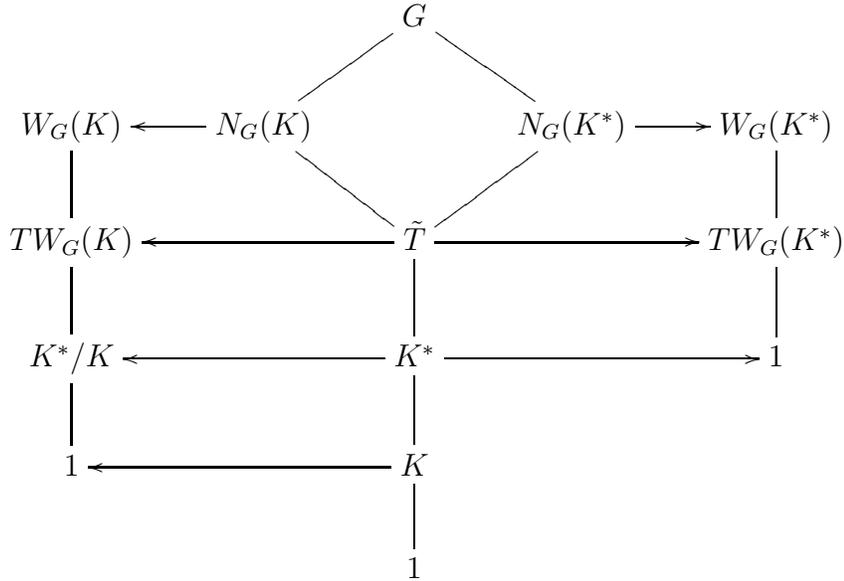
\xymatrix{
    &&G\ar@{-}[dr] \ar@{-}[dl]
    && \\
W_G(K) \ar@{-}[d]&N_G(K) \ar@{-}[dr]\lto& & N_G(\Gp) \ar@{-}[dl]\rto&W_G(\Gp) \ar@{-}[d]\\ 
TW_G(K) \ar@{-}[d]&&\Tt\ar@{-}[d]\rrto \llto&& TW_G(\Gp) \ar@{-}[d]\\ 
\Gp /K \ar@{-}[d]&&\Gp\ar@{-}[d]\llto \rrto&&1\\
1&&K \ar@{-}[d]\llto&&&\\
&&1&&
}$$

Now if $g$ is an element of $G$ whose image in $W_G(K)$ normalizes $TW_G(K)$, $g$
normalizes $\Tt$, and hence its image in $W_G(\Gp )$ normalizes $TW_G(\Gp)$.

Writing $\sim$ for  Borel's rational cohomology isomorphism $N_G(T)\sim G$, 
we have a map $\alpha$
$$W_G(\Gp)\sim N_{W_G(\Gp)}(TW_G(\Gp))\stackrel{\alpha}\lla N_{W_G(K)}(TW_G(K))\sim W_G(K).$$ 
$$\alpha^*: H^*(BW_G(\Gp))\lra H^*(BW_G(K))$$
as required. 
\subsection{Localization}
In the geometric context we take
$$\cE_{K}=\{e(W)\st W\in \Rep(N_G(K)), W^K=0\}.$$
$$\cI_{K}=\{1, e_K\}, $$
and on $\cKR$ we take the trivial multiplicative set.

\section{The abelian models}
\label{sec:abelian}
In Subsection \ref{subsec:standard} we describe the standard model
$\cA (\cV, \mcR, \cW, \cS)$ formed from the space $\cV=\cK\cup \{\Gp\}$
and its auxiliary data $(\mcR, \cW, \cS)$. In Subsections \ref{subsec:presep} and
\ref{subsec:adj} we describe a stalkwise construction and
the relation to the standard model. In Subsection
\ref{subsec:twosheaves} we name two cases of somewhat different
characters, based on the dimension of the stalks of the sheaf $\mcR$.
In Subsection \ref{subsec:homalgstandard}
we construct injective
resolutions in the standard model, showing it is of injective
dimension 1. Finally, in Subsection \ref{subsec:examples} we make this
explicit in some familiar examples (some readers may wish to flick
forward to the examples as they read). 

\subsection{The standard model}
\label{subsec:standard}
Given $\cV$, a sheaf of rings, with Euler classes and a compatible component structure we
may describe the standard model $\cA (\cV, \mcR, \cS,  \cW)$.

We take $\cOcK=\prod_{F\in \cK}\mcR(F)$, and we form the  
multiplicatively closed set
$$\cS :=\{ (s_F) \st s_F\in \cS_{\Gp/F}  \mbox{ and } s_F=1 \mbox{ almost
  everywhere }  \} \subseteq \cOcK.$$ 

Now we consider the diagram of rings 
$$\cO^{\cospan}:=\left(
  \begin{gathered}
    \xymatrix{
&\mcR(\Gp) \dto\\
\cOcK\rto &\cSi \cOcK 
}
\end{gathered}
\right)$$

The standard model $\cA=\cA (\fX,\mcR,\cS, \cW)$ consists of $\cO^{\cospan}$-modules 
$$\xymatrix{
&V\dto\\
N\rto &P 
}$$
where (1) the $\mcR(\Gp)$-module $V$ is torsion, (2) (quasicoherence) the 
horizontal induces an isomorphism $\cSi N\cong P$ and (3)
(extendedness) the vertical
induces an isomorphism $\cSi \cOcK \tensor V\cong P$.

We refer to $N$ as the {\em nub} and $V$ as the {\em vertex}, and informally we 
write $N\lra \cSi \cOcK\tensor V$ for the above object.

There are compatible actions of the finite groups. Thus the $F$th 
component of each element 
of $\cS$ is  $\cW_F$ invariant, there is an action 
of $\cW_F$ on $e_FN$, and of $\cW_{\Gp}$ on $V$, and the map $V\lra \cSi N\lra 
\cSi e_FN$ is $\cW_F$-equivariant.


 \subsection{Separating subgroups}
\label{subsec:presep}
 As described in \cite{adelic1} one expects two different models where
the subgroups are considered separately (the `separated model' where
there is still a weak condition on the vertical map, and the `complete
model', where the $\cOcK$-module is complete in a suitable sense but
there is no condition on the vertical map). We begin by describing 
functors for separating and recombining the subgroups. 

  The pre-separated model $\cA_{ps}=\cA_{ps} (\fX,\mcR, \cS, \cW)$ consists of
  a torsion  $\mcR(\Gp)$-module  $V$ with an action of $\cW_{\Gp}$  and $\mcR(F)$-modules $N(F)$ with an
action of $\cW_F$ for each $F$ together with a spreading map 
$$\sigma: V\lra \cSi \prod_FN(F). $$
We require that the map $V\lra \cSi \prod_F N(F)\lra \cSi N(F)$ is
$W_F$-equivariant.

\subsection{Adjunction}
\label{subsec:adj}
There is an adjunction 
$$\adjunction{e}{\cA}{\cA_{ps}}{p}$$
relating the standard and preseparated models. The constructions we
describe are all consistent with the actions of the component
structure, so actions will not be  mentioned explicitly.

If $X=(N\lra \cSi 
\cOcK\tensor V)$ then $eX$ has the same vertex $V$ and the separated
stalk at $F$ is $N(F)=e_FN$ where $e_F$ is the idempotent supported at
$F$. The 
structure map is the composite 
$$V\lra \cSi \cOcK\tensor V=\cSi N\lra 
\cSi \prod_FN(F) . $$
If $Y=(\{N(F) \}_F, V, \sigma)$ then $pY$ has the same vertex $V$ and 
the nub $N$ is defined by the pullback square 
$$\xymatrix{
  N \rto \dto &\cSi \cOcK\tensor V\dto\\
\prod_FN(F) \rto &\cSi \prod_F N(F) 
  }$$
 Since $\Gamma_{\cS}M\simeq \bigoplus_F e_F\Gamma_{\cS}M$,  
the unit gives an isomorphism $X\stackrel{\cong}\lra peX$. 

  On the other hand the counit need not be an isomorphism. For 
  example,  if $Y$ has the property that $N(F)=0$ for all $F\in \cK$, we find 
  $pY$ has nub $N=\cSi \cOcK \tensor V$, and $epY$ has stalk 
  $\cSi \mcR(F)\tensor V$ at $F\in \cK$, which need not be zero.

  \subsection{Two structure sheaves}
  \label{subsec:twosheaves}
Finally we consider two structure sheaves with somewhat different
behaviours. These correspond to the two cases $\cK=\cK_1$ and 
$\cK=\cK_0$. The present account focuses on a particularly simple
choice of sheaf $\mcR$ that covers examples of immediate interest. 

\begin{defn}
The {\em Type 0} structure sheaf has $\mcR(\Gp)=\Q$ and 
 $\mcR(F)=\Q$ for all $F\in \cK$ and $\cS$ 
is the   multiplicatively 
closed set of functions $\cK \lra \{0,1\}$ with finite support and 
$\cK$ is infinite. 

The {\em Type 1} structure sheaf has $\mcR (\Gp)=\Q$ and 
$R(F)=\Q[c]$ for all $F$ and $\cS$ is the multiplicatively 
  closed set of functions $e:\cK \lra \Z_{\geq 0}$ with finite support. 
\end{defn}

The crudest difference between the types  is about whether the standard and separated
models are essentially different. 

\begin{lemma}
For the Type 0 structure sheaf,  the $p-e$ adjunction is an equivalence
of categories. 
\end{lemma}

\begin{proof}
The statement about Type 0 structure sheaves follows since $\cSi
R(F)=0$, so that applying $e_F$ to the defining pullback for $p$ we
recover $(epN)(F)=N(F)$.
\end{proof}

To explain the situation with Type 1 structure sheaf we need to recall
the models of the generating basic cells
$$\sigma_G=(\cOcK \lra \cSi \cOcK\tensor \Q) \mbox{ and
}\sigma_F=f_F(\Q)=
(\Q_F \lra 0). $$

\begin{lemma}
The Type 1 structure sheaf the counit of the adjunction is not an 
equivalence, but it is a cellular equivalence in the sense that it
induces an isomorphism of $\Hom
(\sigma_G, \cdot)$ and $\Hom(\sigma_F, \cdot)$ 
\end{lemma}

\begin{proof}
With Type 1 structure sheaves, 
the counit is not an isomorphism in the given example of a skyscraper 
sheaf at $G$.

For the cells $\sigma_F$ the cellular equivalence is clear since $\Hom (\sigma_F,
X)=\Gamma_c e_FN$, which does not change under the functor $p$. For
the cell $\sigma_G$ we note that $\Hom (\sigma_G, X )$ is calculated
as  a pullback square:
$$\xymatrix{
  \Hom(\sigma_G, X)\rto \dto &\Hom (\Q , V)\dto \ar@{=}[r]&V\ddto \\
\Hom_{\cOcK}(\cOcK, N)\rto \ar@{=}[d]&\Hom_{\cOcK}(\cOcK, \cSi \cOcK
\tensor V)   \ar@{=}[dr]&\\
  N\rrto&&\cSi \cOcK\tensor V}. $$

In other words the comparison map is the map relating to the inner and
outer pullbacks in the diagram
$$\xymatrix{
  &V\dto \\
  N\rto \dto &\cSi \cOcK\tensor V\dto \\
  \prod_FN(F) \rto &\cSi \prod_FN(F)
}$$
Since the square is a pullback by definition, the map is an
isomorphism. 
\end{proof}
\subsection{Homological algebra of the standard model}
\label{subsec:homalgstandard}
It is an easy exercise to understand this abelian category, but it is  useful to
work it through for reference in higher dimensional contexts. By
Maschke's Theorem the actions of the component groups $\cW_K$ do not
affect the homological dimension. 

\begin{lemma}
\label{lem:dimoneinjdim}  
If the structure sheaf is of Type 0 or Type 1, 
the standard abelian category $\cA$  is of injective dimension 1. 
\end{lemma}
\begin{remark}
(a) One point of writing this down is to highlight the slightly different 
formal structure of the proof in Type 0 and Type 1. 

(b) In Remark \ref{rem:injdim} below that the argument is easily adapted to give an
estimate of the injective dimension when $\mcR$ is a more general
diagram of polynomial rings, but we take the opportunity to be a bit more precise in the
present case. 
\end{remark}

\begin{proof}
  We may write down enough injectives. Indeed, for any graded
  $\Q$-vector space $W$ with an action of $\cW_{\Gp}$, the object 
  $$e(W)=
  \left(
    \begin{gathered}\xymatrix{
&W\dto\\
\cSi \cOcK \tensor W\rto &\cSi \cOcK\tensor W 
}
\end{gathered}
\right) 
$$
has the property $\Hom (X, e(W))=\Hom (V,W)$. It is therefore 
injective. 

Similarly if $T_F$ is a torsion $R(F)$-module, the object 
  $$f_F(T_F)=
  \left(
    \begin{gathered}\xymatrix{
&0\dto\\
T_F \rto &0 
}
\end{gathered}
\right) 
$$
lies in the standard model and has the property 
$$\Hom (X, f_F(T_F))=\Hom (V(F), T_F).$$
It is therefore injective if $T_F$ is an injective $\Q [c]$-module. 

Now suppose $X=(N\lra \cSi \cOcK \tensor V)$. 
 For an arbitrary object $X$ we may take the map $X\lra e(V)$
 corresponding to the identity on $V$; its kernel is at the nub, where
 it is a module $T$ with $\cSi T=0$. In either Type 0 or Type 1, this
 means that $T=\bigoplus_F e_FT$ (direct sum!). In Type 0, $e_FT =I_F$ is
 already injective, and in Type 1 we choose a resolution $0\lra
 e_FT\lra I_F\lra J_F\lra 0$. In any case   we may construct a monomorphism 
$$\mu: X \lra e(V)\oplus \prod_F f_F(I_F).  $$

It is at this point that the two cases differ. 

In the Type 0 case,  we consider the product $\prod_F f_F(I_F)$: the 
value at $\Gp$ is $\cSi \prod_FI_F$, which is non-zero if infinitely 
many terms $I_F$ are nonzero. The cokernel of $\mu$ is thus $e(W)$ for some 
vector space $W$, and we have found an injective resolution of length 1. We observe 
that not all objects are injective since 
 $\Ext^1(e(\Q), \underline{\Q})=\cSi \prod_F\Q\neq 0$.

 In the Type 1 case,  we may again proceed with the resolution, but
as in Type 0,  $\prod_F f_F(I_F)$ may be non-zero at $\Gp$. We may
instead embed $T$ in the injective $\cE$-torsion module $\bigoplus_F f_F(I_F)$.
This  sum of 
 injectives is also injective, since $\Ext (A, \bigoplus_FI_F)=0$ for 
 all $\cOcK$-modules $A$ which occur as a nub (as in 
 \cite[5.3.1]{s1q}).  This is clear for torsion modules $A$, and it is
 clear for $A=\cOcK$, and follows in general from this.

Accordingly, we have 
found an injective resolution of length 1. It is easy to see that 
the are non-split extensions (for example the short exact sequence of 
torsion $R(F)$-modules $0\lra\Q\lra  \Q[c]^{\vee}
\lra \Sigma^{-2}\Q[c]^{\vee}\lra 0$ gives non-split extensions on 
applying $f_F(\cdot)$ for any $F$). 
  \end{proof}

\begin{remark}
We could have used the same argument for  Type 0 as for Type 1, and we
would still have reached the conclusion that the category is of injective
dimension 1. The point is that to use that argument we need to think
in terms of modules over  the ring $\cOcK$ (rather than sheaves over
$\cK^*$), and the ring is not of injective dimension 0. The given
argument  relying on the fact that stalks are
fields seemed more transparent. 
\end{remark}

It is useful to record the following criterion for injectivity for 
later reference.  
\begin{lemma} 
In the Type 0 case, any object $X=(N\lra \cSi \cOcK\tensor V)$ for
which only finitely many of the modules $T_H=e_H(\Gamma_{\cS}N)$ are
non-zero is injective. 

In the Type 1 case, a module $X=(N\lra \cSi \cOcK\tensor V)$ with the
property that 
$e_FN$ is divisible for all $F\in \cK$ is 
injective. 
\end{lemma}

\begin{proof}
In the Type 0 case, only finitely many of the terms $f_F(T_F)$ are
non-zero, so the product of them has zero vertex. 
 
In the Type 1 case, the map $X\lra e(V)\oplus \bigoplus_Ff_F(I_F)$
constructed in the previous lemma is an isomorphism. The point is that the kernel $T$ is injective and so
we can take $I_F=T_F$ and $J_F=0$. 
  \end{proof}

\begin{remark}
\label{rem:injdim}
Essentially the same argument will show that the algebraic model is of finite 
injective dimension if the rings $\mcR (\Gp)$ and $\mcR(K)$ are 
polynomial rings. Indeed, the full subcategory of objects $e(W)$ is
equivalent to $\mcR (\Gp)$ modules. For an arbitrary object $X$ we
consider the map $X\lra e(W)$ where $W$ is the value of $X$ at $\Gp$
and then obtain two exact sequences
$$0\lra K\lra X\lra \overline{X}\lra 0 \mbox{ and } 0\lra
\overline{X}\lra e(W)\lra C\lra 0$$
where $K=f(K')$ and $C=f(C')$ for torsion $\cOcK$-modules $K',C'$. 

If $\mcR(\Gp)$ is a polynomial ring on $a$ variables and $\mcR(K)$ has
at most $b$ variables for $K\in \cK$ then $\injdim (f(C')), \injdim
(f(K')) \leq b$ and hence $\injdim(\overline{X})\leq \max(a, b+1)$,
and therefore the same bound applies to $\injdim(X)$.
\end{remark}

  \subsection{Examples}
  \label{subsec:examples}

  First,  there are three examples where the answer is well known
  \cite{s1q, o2q}. 

  \begin{example}
(i) The very simplest case comes from the circle group $G=T$. This has
a single block, and we take $\cK=\cC$ to be the set 
of finite cyclic subgroups (in bijection to the positive integers by order of 
subgroup). 

Then $\fX_T=\fX (\cC)$, we take the Type 1 structure sheaf
$\mcR(F)=\Q[c]$ for all $F$, and the 
component structure is trivial ($W_F=W_G=1$).  The set $\cS$ consists 
of Euler classes of representations $V$ with $V^T=0$, so that 
$e(V)(F)=c^{\dim (V^F)}$. 

The model $\cA$ is the standard model $\cA (T)$ for rational $T$-spectra. 

(ii) Next we may take $G=O(2)$ and look at the toral block, again taking $\cK=\cC$. Now the 
compactification again consists of closed subgroups of the circle 
$SO(2)$. (The compactifying point $\Gp$ in this example is $SO(2)$). The rings 
are Type 1 as in the previous example, 
but we take each of the groups $\cW_F$ and $\cW_G$ to be of 
order 2 and the maps $\cW_F\lra \cW_G$ to be isomorphisms. The action of 
$\cW_F$ on $\Q[c]$ takes $c$ to $-c$. 

The model $\cA$ is the model $\cA (O(2)|toral)=\cA (SO(2)][W]$ for the toral 
block of rational $O(2)$-spectra. 

(iii) The model for the toral block of  $G=Pin(2)$ is precisely the 
same as in Part (ii). 
\end{example}

For the other cases we will consider only the full subgroups. 

\begin{example}
(i) For $G=T\times C_2$. This has two blocks. The first consists of 
subgroups of $T$, and the second (considered here) consists of full
subgroups.  We  take $\cK$ to consist of the full 
subgroups of $T\times C_2$ and $\Gp=G$ is the compactifying point. 
The subgroups in $\cK$  are either $C_n \times C_2$ or else 
the cyclic groups $C_{2n}'$ generated by $(e^{2\pi i/2n}, -1)$. 

This example is another Type 1 example,
essentially like $\cA (T)$ except that $\cC$ has been 
replaced by a different countable set.

(ii) For $G=O(2)$ the space of subgroups again divides into two
blocks. The first consists of 
subgroups of $SO(2)$, and the second (considered here) consists of
full subgroups.  We take $\cK=\mcD$ to consist of the conjugacy classes of 
dihedral subgroups and $\Gp=G$. This time the ring is Type 0, with $\mcR(F)=\Q$ for all $F$; 
the groups $\cW_F$ are all of order 2, and $\cW_G$ is trivial. 

The multiplicatively closed set $\cS$ consists of the characteristic 
functions of the cofinite sets of $\mcD$.

The model $\cA$ is the model $\cA (O(2)|\full)$ for the dihedral component of 
rational $O(2)$-spectra.  

(iii) For $G=Pin(2)$ the model is essentially the same as that for
$O(2)$. The toral block is identical to that of $O(2)$ and the
block of full subgroups is like that for except that $\cK=\cQ$ consists of the quaternion subgroups. 
\end{example}

Put together we have the following algebraic models for 1-dimensional groups. 
$$\cA(T\times W)\simeq \cA (T)[W]\times \cA (T)$$
$$\cA(O(2))\simeq \cA (SO(2))[W] \times W\mbox{-Sh}/(\mcD^{\hash})$$
$$\cA(Pin(2))\simeq \cA (Spin(2))[W] \times W\mbox{-Sh}/(\cQ^{\hash})$$
There is an equivalence $\cA (O(2))\simeq \cA (Pin(2))$, but 
note this is given by different bijections on the two components (in 
the non-toral part it is natural to choose the bijection 
$Q_{4a} \leftrightarrow D_{2a}$ quotienting out the central subgroup of order 2, but in the 
toral part we must choose a bijection between cyclic subgroups of 
$SO(2)$ and those of $Spin(2)$, so we cannot use the quotient map).

\begin{example}
(i)  We may consider the block corresponding to full subgroups of the normalizer in the 
maximal torus in $SU(3)$. This consists of $\cW$-sheaves over $\cK^\hash$
where $\cK$ is the discrete space of conjugacy classes of finite 
subgroups and $\cW_F$ is a group of order 3 for all $F$ and $\cW_{\Gp}=1$. 
In this case the full subgroups are $T^2[n]\sdr 
\Sigma_3$, and therefore in bijection with the positive integers (see
\cite[Section 13]{t2wqalg} for more details). 

(ii)  We may consider the block corresponding to full subgroups of 
$T^2\sdr C_3$ (a subgroup of $T^2\sdr \Sigma_3$).  Again
$\cW$-sheaves over  $\cK_3^\hash$
where $\cK_3$ is the discrete space of conjugacy classes of finite 
subgroups and agan $\cW_F$ is of order 3 for all $F$ and $\cW_{\Gp}=1$. 
The set $\cK_3$ is described in \cite[Section 12]{t2wqalg}. 
  \end{example}

\part{Topology}
\section{The abelian models are Quillen models: general strategy}
\label{sec:genstrat}
We will show that the abelian categories provide models in all the
cases we study. The 
structure of the argument is the same as that for tori in 
\cite{tnqcore}: we show that the  sphere spectrum is the pullback of 
rings which are isotropically concentrated and formal in a strong 
sense. 

\subsection{Outline}
The core of the proof is the fact that  the sphere spectrum $S^0$ is a pullback of an
isotropic cube,  using the general 
inductive argument of \cite[8.1]{adelicm}, adapted to the
non-Noetherian setting.

This allows us to outline the proof: we give a symbolic description 
and then explain the notation and discuss the ingredients in the
argument. Details will be given in the rest of the paper.      
      \begin{multline*}
        \Gspectra|\cV 
\stackrel{0}\simeq S^0\module- \Gspectra|\cV 
      \stackrel{1}\simeq   \cell ((S^0)^{\cospan}\module -\Gspectra)\\
    \stackrel{2}  \simeq  \cell ((\widetilde{(S^0)}^{\cospan})^{G}\module-\cW-\spectra) 
    \stackrel{3}  \simeq  \cell (C^{G}_*(\widetilde{(S^0)}^{\cospan})\module 
    -\Q [\cW] \module)\\
    \stackrel{4}  \simeq  \cell (\pi^G_*(\widetilde{(S^0)}^{\cospan})\module -\Q [\cW]\module) 
     \stackrel{5} \simeq DG-\cA (G|\cV) 
   \end{multline*}

Equivalence 0 simply uses the fact that $G$-spectra are 
modules over the sphere spectrum. Equivalence 1 uses the fact 
\cite[4.1]{diagrams} that the category of modules over a homotopy pullback ring is equivalent to the 
cellularization of the category of generalized diagrams over the 
individual modules, together with the fact that the localized sphere 
is the pullback. 

From Equivalence 2 onwards,  we introduce variants on the terms of the initial cube
so as to keep track of  the finite Weyl groups. The cospan
$(\widetilde{S^0})^\cospan$ of $G$-spectra replaces terms of
$(S^0)^\cospan$ by coinductions which vary by subgroup. 
The new objects have  the property that their $G$-fixed point spectra are
products of spectra with homology $H^*(BW_G^e(K))$ for relevant
subgroups $K$ and the category of modules take values in the corresponding product of
categories with $W_G^d(K)$-action. The cellularization will pick out
the appropriate abelian category. 

Equivalence 2 uses the results of  \cite{modfps}. To explain,   for each subgroup $K$
we consider the normalizer $N=N_G(K)$, the Weyl group $W=N_G(K)/K$
with identity component $W^e$ and discrete quotient $W^d$. Finally, we
write $N^f$ for the inverse image of $W^e$ in $N$. With this notation, 
there are equivalences
\begin{multline*}
  DE\lr{K}\module-G-spectra \simeq 
DE\lr{K}\module-N-spectra \simeq \\
DE\lr{K}^{K}\module-W-spectra\simeq
DE\lr{K}^{N^f}\module-W^d-spectra
\end{multline*}
where the first is the forgetful map and the second is passage to $K$
fixed ponts (an equivalence because $DE\lr{K}$ lies over $K$) 
and the third is passage to $N^f$-fixed points under $W^e$ (an
equivalence by the Eilenberg-Moore theorem because $W^e$ is connected).
It requires some care to assemble these equivalences when $\cK$ is
infinite, and we will explain in Section \ref{sec:unif}. 

Equivalence 3 follows from Shipley's Theorem \cite{ShipleyHZ}, and is 
easily adapted to the type of diagram we have. Equivalence 4 is a formality 
statment. Finally, Equivalence 5 
folows from the Cellular Skeleton Theorem, which will 
identify the cellularization of the algebraic category of modules with 
the derived category of an abelian category. 

We will first expla The abelian models are Q models gen star in the argument for individual subgroups and then
discuss how to assemble these for the whole category.

\section{Modules over completions and completions of modules}
\label{sec:DEGp}
In this section we consider the stalks over a single subgroup. We deal
with two particular matters we address. The first is reflected in the title:
the splicing data comes about because we have models for
arbitrary modules over the completed rings, not just complete
modules. The primitive example to bear in mind is that 
there many more modules over $\Zpadic$ 
than there are complete modules (for example $\Zpadic[1/p]$) . Roughly speaking the complete modules are the ingredients but 
modules over the completed ring are used in the splicing. 
The second matter is that we clarify the necessary generators when the
group is not connected in a way that will be important when we allow
infinitely many subgroups.

\subsection{Trivial coisotropy (connected)}
Starting with the simplest case we consider a connected compact Lie
group $G$ and focus on trivial isotropy or coisotropy.
First of all, there is an equivalence
$$\freeGspectra\simeq \cofreeGspectra$$
using the functors $(\cdot )\sm EG_+$ and $F(EG_+, \cdot)$. 
Similarly in algebra
$$\mbox{torsion-$H^*(BG)$-modules}\simeq
\mbox{complete-$H^*(BG)$-modules}$$
using (derived) torsion and derived completion functors
$\Gamma_I(\cdot)$ and $\Lambda_I(\cdot)$. (A model for
torsion $H^*(BG)$-modules is given by DG modules in the abelian
category of torsion modules. In the complete case, the category of
$I$-adically complete modules is not abelian, so one needs to use the abelian
category of $L_0^I$-complete modules \cite{PWcomp}.)

We have equivalences
$$\xymatrix{
\freeGspectra \rto^{\simeq}\dto^{\simeq} &\cofreeGspectra\dto^{\simeq}\\
\torsionHBGmodules \rto^{\simeq}& \completeHBGmodules 
}$$
We will focus on the right hand end (cofree spectra and complete
modules).

One method of proving the vertical equivalences is to observe that
$S^0\lra DEG_+$ is a non-equivariant equivalence, and hence 
$$\cofreeGspectra =L_{EG_+}(S^0\moduleGspectra)\simeq
L_{EG_+}(DEG_+\moduleGspectra). $$
We may then prove $DEG_+$ is formal. 
\begin{prop}
We have equivalences
$$DEG_+\moduleGspectra \simeq DBG_+\modulespectra\simeq
C^*(BG)\modules\simeq \HBGmodules$$
\end{prop}

\begin{proof}
Writing $R=DEG_+$, the first equivalence is in \cite{modfps}, using the fixed point
functor
 $R\moduleGspectra\lra R^G\modulespectra$ and its left adjoint. 
This is an equivalence by the Cellularization Principle
\cite{cellprin} provided $R$-modules are generated by $R$.

This is true when $G$ is connected (there are various proofs, but
one giving this generality is in \cite{bgen}). 
\end{proof}

\begin{thm}
If $G$ is connected, we have a commutative square 
$$\xymatrix{
DEG_+\moduleGspectra\dto^{\simeq}\rto &
\cofreeGspectra\dto^{\simeq}\\
\HBGmodules \rto &\completeHBGmodules 
}$$
\end{thm}

\begin{proof}
In summary, we have the equivalences 
$$\xymatrix{
DEG_+\moduleGspectra\dto^{\simeq}\rto 
&L_{EG_+}(DEG_+\moduleGspectra)\dto^{\simeq}\rto &
\cofreeGspectra\dto^{\simeq}\\
DBG_+\moduleGspectra\dto^{\simeq}\rto &L_{S^0}(DBG_+\moduleGspectra)\dto^{\simeq}\rto &
(\cofreeGspectra)^G\dto^{\simeq}\\
\HBGmodules \rto &L_{\Q}\HBGmodules \rto &
\completeHBGmodules 
}$$
\end{proof}

\subsection{Trivial coisotropy (disconnected)}

We explain how to cover the case of disconnected groups.

\begin{thm}
For an arbitrary group $G$ we have a commutative square 
$$\xymatrix{
DEG_+\moduleGspectra\dto^{\simeq}\rto &
\cofreeGspectra\dto^{\simeq}\\
H^*(BG_e)[G_d]\modules \rto &\mbox{complete-}H^*(BG_e)[G_d]\modules 
}$$
\end{thm}
\begin{proof}
From the Eilenberg-Moore spectral sequence $DEG_+$-modules are
generated by $G^d/F_+$-modules where $F$ is a finite subgroup of
$G^d$. The map of $G_d$-spectra $S^0\lra eS^0$ is a non-equivariant
equivalence where $e\in A(G_d)$ is the idempotent supported at 1. Thus
$S^0\sm DEG_+\lra eS^0\sm DEG_+$ is an equivalence. The result follows
from the fact that free $G^d$-spectra are generated by $G^d_+$.
  \end{proof}

  \section{Treating infinitely many subgroups at once}
\label{sec:unif}
  In the previous section, we argued that $DEG_+$ is strongly isotropically formal
  in the sense that we could take fixed points for a single subgroup
  and obtain a formal ring spectrum. It is rather easy to adapt this
  to $DE\lr{H}$ for any single subgroup $H$, by the method of
  \cite{specgq}. We recall this argument below, but the main point of
  the present section is  to explain how to deal with infinite
  products of such spectra.

 In the previous
  section we explained how to deal with a single subgroup, which we
  took to be the trivial group for convenience. The argument involved
  a subgroup $K$ and some associated subgroups. Principally this means
  its normalizer $N=N_G(K)$, and its Weyl group $W=N/K$, but we also
  need to mention the identity component $W^e$ of $W$, and its
  discrete quotient $W^d=W/W^e$, and finally $N^f$, which is the
  inverse image of $W^e$ in $N$, so that $N/N^f\cong W/W^e=W^d$.

  The argument is that (for suitable ring spectra $R$) we have
  equivalences
  $$R\modules-G-\spectra\stackrel{(1)}
  \simeq R\modules-N-\spectra\stackrel{(2)}\simeq R^K\modules -W-\spectra
  \stackrel{(3)}\simeq R^{N^f}\modules-W^d-\spectra$$
Equivalence (1) is the forgetful map from $G$-spectra to
$N$-spectra, and relies on fusion of
$N$-conjugacy classes being favourable.

Equivalence (2) is passage to $K$-fixed points, and relies on  $R$ having
geometric isotropy consisting of subgroups containing $K$. Equivalence (3) is passage to
categorical fixed points, and uses the Eilenberg-Moore theorem for the
{\em connected} group $W^e$.

We now wish to treat many subgroups $K$ at once, but the intermediate
categories and functors in the above argument involve $N=N_G(K)$ and
therefore depend on $K$ . We explain here that we may instead
factorize the composite so that the dependency on $K$ occurs in the
ring and whilst the categories and functors are independent of $K$.

The subtlety  is that both $N$ and $N^f$ play a role. In our examples
$R\simeq F_{N}(G_+, R)$ but $R\not\simeq F_{N^f}(G_+,R)$ (consider
$K=D_{2n}$ insider $G=O(2)$ (or even inside $G=D_{4n}$)). Thus we need
to use $N$. On the other hand, taking $N$-fixed points loses the
action of the $W^d$. The solution is to reinsert the action of $W^d$
as an endomorphism of the functor $F_{N^f}(G_+, \cdot)$.

\begin{lemma}
\label{lem:uniform} 
 The diagram
  $$\xymatrix{
    R\modules-G-\spectra \rto
    \dto^{\res^G_N}&F_{N}(G_+,R)\modules-G-\spectra 
    \rto&F_{N^f}(G_+,R)\modules-G-\spectra [W^d] \ddto^{\Psi^G}\\
    R\modules-N-\spectra \dto_{\Psi^K}\drto^{\Psi^{N^f}}&&\\
    R^K\modules-W-\spectra \rto_{\Psi^{W^e}}
    &R^{N^f}\modules-W^d-\spectra\rto^{\simeq}&
    R^{N^f}\modules-\spectra [W^d]
    }. $$
  commutes where the top horizontal is $F_{N^f}(G_+, \bullet)$.  The
  composite is a right Quillen functor. 
  \end{lemma}

  \begin{proof}
 The commutativity is just the formula
 $$F_{N^f}(G_+, X)^G\simeq X^{N^f}. $$
All functors in the diagram are  right adjoints. 
    \end{proof}

The maps in the diagram are equivalences under various circumstances. 
  
    \begin{lemma}
      (i) The restriction functor
      $$R\moduleGspectra \lr{\cV} \lra
      R\moduleNspectra\lr{\cV}$$
      is an equivalence if, for all $K\in \cV
\cap \supp (R)$ a containment ${}^gK\subseteq N$ implies $g\in N$, so that
$(G/N)^K$ is a singleton.

(ii) The map
$$F_{N}(G_+,R)\modules-G-\spectra \lra
F_{N^f}(G_+,R)\modules-G-\spectra [W^d] $$
is an equivalence if $R$ is $N^f$-free.

(iii) The map
$$\Psi^K: R\modules-N-\spectra \lr{\cV} \lra R^K\modules-W-\spectra \lr{\cV}$$
is an equivalence if every element of $\supp(R) \cap \cV$ contains $K$.

(iv) The map 
$$\Psi^{W^e}: R^K\modules-W-\spectra \lr{\cV} \lra 
R^{N^f}\modules-W^d-\spectra \lr{\cV}$$
is an equivalence if $R^{N^f}$ is $W^e$-free and 
$W^d_+\sm R$ generates $R^K\modules-W-\spectra$. 
\end{lemma}

\begin{proof}
Part (i) follows since $G/N_+\sm M\simeq G_+\sm_N M\lra M$ is an 
equivalence for all $R$-modules $M$. 

Part (ii) is the fact that free $W^d$-spectra are generated by $W^d_+$
and hence  by Morita theory equivalent to spectra with a $W^d$-action.

Part (iii) is  \cite[Theorem 7.1]{modfps}. 

Part (iv) follows from \cite{gfreeq2}. 
  \end{proof}



\section{The abelian models are Quillen models in dimension 1}
\label{sec:dimonemodels}
In this section we prove a general theorem. Amongst the small cases
covered are the following familiar examples  
\begin{itemize}
\item $G=T$, the circle group 
\item $G=O(2)$, toral subgroups $\cV^{O(2)}_1$
\item $G=O(2)$, full subgroups $\cV^{O(2)}_W$
  \item $G=Pin(2)$, toral subgroups $\cV^{Pin(2)}_1$
  \item $G=Pin(2)$, full subgroups $\cV^{Pin(2)}_W$
    \item $G=T\times C_2$, toral subgroups $\cV^{T\times C_2}_1$
\item $G=T\times C_2$, full subgroups $\cV^{T\times C_2}_W$
\item $G$ of rank 2 and $W=C_3$ or $\Sigma_3$ and $\Lambda_0$
  non-trivial. 
  \end{itemize}
We saw in Lemma \ref{lem:dimoneinjdim} that $\cA (G|\cV)$ is of finite
injective  dimension in these cases and hence $DG-\cA (G|\cV)$ admits 
the injective model structure with homology isomorphisms as weak 
equivalences. 

\begin{thm}
  Suppose $G$ is a compact Lie group and  $\cV$ is obtained from a set $\cK$ of conjugacy classes of finite
  subgroups of $G$ by adjoining the conjugacy class of a subgroup
  $\Gp$ with finite Weyl group. 
  Then there is a  Quillen equivalence 
  $$\Gspectra|\cV\simeq DG-\cA (G|\cV). $$
 where 
 $$\cA (G|\cV)\simeq \cA (\cK, \mcR, \cS, \cW), $$
 where the data is as in Section \ref{sec:sheafdatatop}
     \end{thm}

Since we know how to deal with 0-dimensional summands, we suppose
$\cK\mcR=\emptyset $. The first step is to express the category as a pullback,
which we do by giving a pullback square of ring spectra. We will then
calculate the homotopy of the ring spectra, and their strong isotropic
formality. Finally we will prove the algebraic category has the model
as a cellular skeleton.

\subsection{The pullback square}
Broadly speaking there are two ways of presenting the argument. (1)  Giving
a pullback square of rings in $G$-spectra  and then localizing to give
a pullback square in $G$-spectra with geometric isotropy $\cV$. (2)
Constructing the local objects directly and showing they form a
pullback square. We will adopt the first option, since we quickly get
a pullback square and stay with familiar objects as long as
possible. When the filtration is more complicated, one can imagine
that (2) may be more attractive.

For our present  case, the pullback square is extremely familiar:
 we begin with the homotopy Tate square
      $$\xymatrix{
        S^0\rto \dto  &\etf \dto\\
D\efp\rto &\etf\sm D\efp        
}$$
where $\cF$ is the family of all finite subgroups. This is a pullback
square in $G$-spectra, and hence also in the category of spectra over
$\cV$. Furthermore, the objects are all commutative rings.  

In the present case it is very easy to find more economical
representatives of the homotopy types in the category of spectra over
$\cV$.

\begin{lemma}
  In the category of $G$-spectra over $\cV$ we have equivalences

  (i) $\etf\simeq E\lr{\Gp}$ and

  (ii) $E\cF_+\simeq \bigvee_{K\in \cK} E\lr{K}=E\lr{\cK}$

   \end{lemma}

    \begin{proof}
   By hypothesis,  $\cV=\cK \amalg \{\Gp\}$ for a collection $\cK$ of finite
   subgroups, not usually a family.

Since, $\Gp$ is the only infinite subgroup in $\cV$, Part (i) is
immediate. 

By \cite[Theorem  1.2]{EFQ}, there is an equivalence
$\efp \simeq \bigvee_{F\in \cF} E\lr{\cF}$. The inclusion of
$E\lr{\cK}$ is a $\cV$-equivalence.
\end{proof}

  \begin{remark}
Using the splitting cited in the lemma, there is an idempotent selfmap $e_{\cK}$ of $\efp$, with $e_{\cK}\efp 
\simeq E\lr{\cK}$. We then have 
$ e_\cK D\efp\simeq \prod_{K\in \cK} DE\lr{K}\simeq DE\lr{\cK}$. 

    The spectra 
    $DE\cF_+$ and $\prod_{K\in \cK} DE\lr{K}$ are usually not 
    equivalent in spectra over $\cV$, but we show their difference is entirely 
    concentrated at $\Gp$. 
    \end{remark}

\begin{cor}
  The square 
$$\xymatrix{
        S^0\rto \dto  &E\lr{\Gp}\dto\\
DE\lr{\cK}\rto &E\lr{\Gp}\sm DE\lr{\cK}        
}$$
is a homotopy pullback in $G$-spectra over $\cV$.
\end{cor}

\begin{proof}
This follows since the horizontal fibres of the original pullback are $\efp$ and
in the second, the fibre of the lower horizontal is $E\lr{\cK}$.
\end{proof}

\subsection{Coefficient rings}
It remains to study  the categories of modules over each of the
terms. For this we adapt the work Sections \ref{sec:DEGp} and \ref{sec:unif}. 
For individual subgroups we have already seen modules over $\mcR(K)
=DE\lr{K}$ in $G$-spectra are equivalent to  modules over
$\mcR(K)^{N^f_G(K)}\simeq C^*(BW_G^e(K))$, which is a non-equivariant
spectrum with an action of $W_G^d(K)$. 

For each $K\in \cK$ we consider the $G$-spectrum
 $\mcRt(K)=F_{N^f_G(K)}(G_+, \mcR (K))$ and then let $R=\prod_{K\in
   \cK}\mcRt(K)$. We have
$$R^G=\left( \prod_KF_{N_G^f(K)}(G_+, \mcR(K))\right)^G\simeq \prod_K\mcR(K))^{N_G^f(K)}
$$
and 
$$\mcR(K)^{N_G^f(K)}=(\mcR(K)^K)^{W_G^e(K)}=(DEN/K_+)^{W_G^e(K)}=D\Bt W_G^e(K). $$

Stripping away specifics, we will argue that $\mcR(K)$-module $G$-spectra are
equivalent to modules over a commutative ring $\pi_*^{N^f}(\mcR(K))$ for a
subgroup $N^f$ of $G$ in a
category of representations of a finite group $W^d$. There are steps
in equivariant homotopy, leading from $\mcR(K)$-modules in $G$-spectra to
$\mcR(K)^{N^f}$-modules in spectra with an action of $W^d$. Shipley's
Theorem shows this is equivalent to $\Q[W^d]$ modules over a DGA  with homology
$\pi_*(R^{N^f})$, and this is shown to be formal.

From an expository point of view, the first thing to identify is the
target ring $\pi_*(\mcR(K)^{N^f})$, even though it doesn't play a role in
the argument until the end. 

\begin{lemma}
\label{lem:coeffs}
  We have 
  $$\piG_*(\mcRt(1))= \piG_*(DE\lr{\cK})=\prod_KH^*(BW_G^e(K))=\cOcK$$
    $$\piG_*(\mcRt(\Gp))=\piG_*(\mcR (\Gp))=\pi^G_*(E\lr{\Gp})=H^*(BW_G^e(\Gp))=\cO_{\Gp}=\Q$$
  \end{lemma}

  \begin{proof}
The first part is   straightforward from the isomorphisms
    \begin{multline*}
 \pi^G_*(DE\lr{K})=[S^0, DE\lr{K}]^G_*=[E\lr{K},S^0]^G=[E\lr{K},
 S^0]^N=\\
 [EN/K_+,S^0]^{N/K}=[BN/K_+,S^0]=H^*(BN/K)
 \end{multline*}

The second part is simple since $\Gp$ has finite Weyl group
(otherwise, we would need to take careful account of the
fact the calculation is in the $\cV$-local category).
\end{proof}

\begin{lemma}
\label{lem:piGcS}
There is a multiplicatively closed set $\cS$ so that 
  $$\pi^G_*(E\lr{\Gp}\sm DE\lr{\cK})=\cSi \prod_{K\in \cK}H^*(BW_G^e(K))=\cSi\cOcK$$
To describe the localization we   split the product up into three 
factors 
$$\cOcK=\cOcKone\times \cOcKzero \times \cOcKR$$
\begin{itemize}
\item On $\cK_1$ the multiplicatively closed set is 
  $\cE =\{e(W)\st W\in \mathrm{Rep}(N_G(\Gp)), W^{\Gp}=0\}$
\item On $\cK_0$, the multiplicatively closed 
set is 
$\cI=\{i(U) \st U\subseteq \cK_0  \mbox{ closed in } \cV \}$. 
\item 
On $\cKR$, the   multiplicatively closed set is $\{ 0\}$. 
\end{itemize}
\end{lemma}

\begin{proof}
The splitting and the $\cKR$ part are clear. 

We start with the $\cK_0$ part. Since supports are closed under
cotoral specialization, any idempotent nonzero on $\Gp$ will
be nonzero on $\cK_1$. Next we note that for any $\cV$-closed 
subset $U$ of $\cK_0$  there is an  idempotent 
$e_U$ with support $U\cup \{\Gp\}\cup \cK_1$ because the latter is open and 
closed in $\cV$.  We may then observe 
$$E\lr{\Gp}\sm E\lr{\cK_0}=\colim_U e_UE\lr{\cK_0},  $$
because $e_UE\lr{K}\simeq *$ for $K\not \in V$. 

\newcommand{\DEK}{DE\lr{\cK}}
For the $\cK_1$ part we will show directly that the algebraic
localization 
$$\DEK \lra \cEi \DEK $$
is the isotropic localization. For this we need to show that $\cEi
\DEK$ is $K$-equivariantly contractible for all $K\in \cK$ and that 
$\Gamma_{\cE}\DEK$ has non-equivariantly contractible geometric
$\Gp$-fixed points. 

For the first statement, it suffices to say that for each $K\in \cK$
there is an element $e\in \cE$ so that $e $ is $K$-equivariantly
null. For this we choose $e(V)$ for a representation of the torus 
$\Gp/K$. For the second statement, we may work $\Gp$-equivariantly,
and we note that inverting $\cE$ commutes with restriction of
subgroups. Thus we have a  $\Gp$-equivariant equivalence 
$\cEi \DEK\simeq \siftyV{\Gp}\sm \DEK$ and $\Gamma_{\cE}\DEK\simeq 
E\cP_+\sm \DEK$, which has not $\Gp$-fixed points as required.  



\end{proof}

\begin{remark}
In the $\cK_1$ part one might hope to make an argument using
representation theory to give a model of $\etf \sm DE\lr{\cK}$ of the
form $\siftyV{\Gp}$. This will often work, for example when $\Gp=G$. 
However if $G=SO(3)$ and $\Gp=SO(2)$, all representations of $SO(3)$
have fixed points under elements of order 2, so no such model exists. 
\end{remark}

\begin{cor}
  \label{cor:piGcS}
  (i) For any $DE\lr{\cK}$-module $N$, the map
  $$N\lra E\lr{\Gp} \sm N$$
  induces inversion of $\cS$, so that 
  $$\piG_*(E\lr{\Gp}\sm N) =\cSi \pi^G_*(N). $$

(ii)  The map $E\lr{\Gp}\sm X\lra DE\lr{\cK} \sm E\lr{\Gp}\sm X$ induces extension of scalars along $\Q=\mcR(\Gp)\lra \cSi \cOcK$
so that 
$$\pi^G_*(E\lr{\Gp}\sm DE\lr{\cK}\sm X)=\cSi \cOcK \tensor
\pi_*(\Phi^{\Gp}X)$$
\end{cor}

\begin{proof}
  (i) We note that $\piG_*=\pi^{N}_*$ for these spectra, and then the
  models of $E\lr{\Gp}$ given in the lemma give the result.

  (ii) The map $S^0\lra DE\lr{\cK}$ induces a map
  $\pi_*(\Phi^{\Gp}X)\lra \pi^G_*(E\lr{\Gp}\sm DE\lr{\cK}\sm X)$ which
  extends to a natural transformation.   Since $\cSi\cOcK$ is flat it
  is a natural transformation of homology theories compatible with the
  action of $\cW_K$.  In checking it is an isomorphism we may ignore
  the action of $\cW_K$. The result therefore follows from the case
  when $X$ is an $N^f$-equivariant sphere. 
\end{proof}

\subsection{Uniformization and formality}
There are two convenient facts about module categories that we now
need to exploit. Firstly, for a finite group $W$, the category of $A[W]$-modules is equivalent
to the category of $A$-modules with a $W$-action: $A[W]\modules\simeq
(A\modules)[W]$ (requires $W$ to be finite).
Secondly, the category of modules over a product
$A_1\times A_2$ is equivalent to the product of the module categories:
$(A_1\times A_2)\modules\simeq (A_1\modules)\times
(A_2\modules)$ (requires that this is a finite product).  

By tom Dieck's finiteness theorem \cite{tomDieckTGRT}, $W_G^d(K)$ takes only finitely many 
values, so there is a partition of $\cK=\cK_1\amalg \cdots \amalg
\cK_N$ into finitely many pieces, each of which has a single value  $W_i$ of $W_G^d(K)$.
For brevity we write 
$$(\prod_KR_K)\modules [\cW]:=(\prod_{K\in
  \cK_1}R_K\modules)[W_1]\times \cdots \times
(\prod_{K\in \cK_N}R_K\modules)[W_N]$$

\begin{lemma}
\label{lem:dek}
  With $R=DE\lr{\cK}$, there is an equivalence
  $$R\modules-G-\spectra \simeq \left(\prod_{K\in \cK}D\Bt
    W_G^e(K)_+\right) \modules-\spectra[\cW].$$
\end{lemma}

\begin{proof}
Using Lemma \ref{lem:uniform}, we have a right Quillen functor 
$$R\modules-G-\spectra\lra ( R^G\modules-\spectra) [\cW]$$
which is an equivalence on each individual factor. 

Breaking the product up into finitely many parts we may assume that
there is a single finite group $W$ associated to each factor $R_K$, so
that each factor is generated by $\mcRt(K)=F_{N_G^f(K)}(G_+, \mcR(K))$
corresponding to $D\Bt W_G^e(K)_+\sm W_+$. The product on the right is 
generated by $(\prod_K D\Bt W_G^e(K)_+)\sm W_+$. 

Equivalences of cells are detected on the factors,  so that the unit
and counit are equivalences on generators. Hence by the Cellularization Principle \cite{cellprin}, it is an equivalence as required. 
\end{proof}

\begin{cor}
(i) The category of $DE\lr{\cK}$-module $\Gspectra \lr{\cV}$ is 
equivalent to the category of DG $\cW$-equivariant objects in 
$\cOcK$-modules.

(ii) The category of $DE\lr{\Gp}$-module $\Gspectra \lr{\cV}$ is 
equivalent to the category of DG $\cW_{\Gp}$-equivariant $\cO_{\Gp}$-modules.
\end{cor}

\begin{proof} Case (ii) is straightforward. 

In Case (i) we have a Quillen pair showing that $DE\lr{\cK}$-module
$\Gspectra \lr{\cV}$ are equivalent to $D\tilde{E}\lr{\cK}^G$-module
$\cW$-spectra by Lemma \ref{lem:dek}. By Shipley's Theorem this is
equivalent to DG modules in $\cW$-spectra over a DGA with homology
$\pi_*D\tilde{E}\lr{\cK}^G=\prod_{K \in \cK} H^*(BW_G^e(K))=\cOcK$ by
Lemma \ref{lem:coeffs}. 

To see that the DGA is formal, we argue as follows. For a single
factor  we have a CDGA $\cO_{K}'$ with homology $\cO_K$. Thus
$H_*(\cO_{\Gp}')=Z(\cO_{\Gp}')/B(\cO_{\Gp}')$ is isomorphic 
to $\cO_{\Gp}=\mcR (\Gp)$, which is the symmetic algebra on a finite 
dimensional vector space $V$ with an action of 
$\cW_\Gp$. By Maschke's Theorem we may find a 
submodule $V'$ of $Z(\cO_{K}')$ mapping to $V$, and thus we may choose an 
equivariant homology isomorphism $\cV=\mathrm{Symm} (V)\lra \cO_{K}'$. This
covers the generic point $\Gp$. For $\cK$, we then use the equivalence
$\cOcK'\lra \prod_K e_K\cOcK'$ followed by the above maps on each
factor. 
\end{proof}

We need to show that the objectwise equivalences may be assembled to
an equivalence of cospans. 

\begin{lemma}
\label{lem:formalityone}
The diagram of $G$-spectra is intrinsically formal. 
\end{lemma}

\begin{proof}
The proof is essentially the same as for the circle. Let us suppose we
have a cospan $\cO'_{\Gp} \lra \cT'\lla \cO'$ with homology $\cO_{\Gp}\lra \cSi
\cOcK\lla \cOcK$.

Starting at $\cO_{\Gp}'$, we know $H_*(\cO_{\Gp}')=Z(\cO_{\Gp}')/B(\cO_{\Gp}')$ is isomorphic 
to $\cO_{\Gp}=\mcR (\Gp)$, which is the symmetic algebra on a finite 
dimensional vector space $V$ with an action of 
$\cW_\Gp$. By Maschke's Theorem we may find a 
submodule $V'$ of $Z(\cO_{\Gp}')$ mapping to $V$, and thus we may choose an 
equivariant homology ismorphism $\cV=\mathrm{Symm} (V)\lra \cO_{\Gp}'$.

 This gives an equivalence
$$\xymatrix{
\cO_{\Gp}'\dto & H^*BW_G^e(H)\dto\lto \\
\cT'\rto &\cT''=\cO_{\Gp}'\tensor_{\cT'}\cT'&\cT''\lto_= \\
\cO'\uto \rto^= &\cO'\uto&\cO'\uto \lto_=
}$$
and we may suppose $\cO_{\Gp}'$ is  standard. Now proceed as in
the diagram below. 

The
original diagram $\cO_{\Gp}\lra \cT \lla \cO$ has homology $\mcR(\Gp) \lra \cSi \cOcK
\lla \cOcK$. In the following diagram, all horizontals are homology
isomorphisms.  Most maps are self explanatory, but we note that
inverting additional classes already inverted in homology induces a
weak equivalence. We therefore choose $\hat{\cS}$ by adding
representative cycles to $\cS$ so the the relevant DGA receives a map.

$$\xymatrix{
\mcR (\Gp)\dto \rto & \mcR (\Gp)\dto &\mcR (\Gp) \lto \rto \dto&\mcR (\Gp)\dto &\mcR (\Gp) \dto\lto\\
\cSi \prod_KH^*(BW_G(K))\rto &\hat{\cS}^{-1} \prod_Ke_K\cO' &\hat{\cS}^{-1} \cO'  \lto\dto \rto  
&\hat{\cS}^{-1} \cT''& \cT''\lto \\
\prod_F H^*(BW_G^e(K)) \uto \rto & \prod_ Ke_K\cO' \uto &\cO'\lto \rto &\cO'\uto
&\cO' \lto \uto  }$$
\end{proof}




\subsection{The cellular skeleton theorem}
We have shown that the category of $G$-spectra over $\cV$ has a purely
algebraic model, which is the cellularization of a category of
modules. The cellularization means that the weak equivalences are not
obvious.  We complete the picture by showing that this
algebraic cellularization is equivalent to DG objects in the abelian
category $\cA (G|\cV)$ with the weak equivalences being homology
isomorphisms.

More precisely, $\cA (G|\cV)$ is a  coreflective subcategory of the category of
modules with cellular equivalences of DG objects being homology
isomorphisms, and the inclusion of abelian categories induces a Quillen equivalence.

\begin{lemma}
\label{lem:CST}
The cellularization of the category of
$\pi^G_*((S^0)^\cospan)$-modules has $\cA (G|\cV)$ as a cellular skeleton.
\end{lemma}

\begin{proof}
  First we note that the images of any $G$-spectrum lie in $\cA
  (G|\cV)$. This is because the homotopy of the horizontal map is
  inverting the multiplicatively closed set $\cS$ by Corollary
  \ref{cor:piGcS} (i),
  and the vertical map is extension of scalars along a flat map by
  Corollary \ref{cor:piGcS} (ii). In particular, the cells are in $\cA (G|\cV)$. 

To see that any object of $\cA(G|\cV) $ is cellular, we use the injective resolutions
from Subsection \ref{subsec:homalgstandard}. The point is that any object of $\cA (G|\cV)$
admits an embedding into an injective which is a sum of
injectives $e(H_*(BW_G^e(\Gp)))$ and $f_K(H_*(BW_G^e(K))$, whose
cokernel is a sum of injectives of the second type. Each of these
injectives are cellular because they are the images of
$G$-spectra. Indeed
$e(H_*(BW_G^e(\Gp))$ is the image of $\etf$, and $f_K(H_*(BW_G^e(K))$
is the image of $DE\lr{K}\sm F_N(G_+, EW_G(K)_+)$.

Finally, we argue that any object in the module category  is
cellularly equivalent to a an object of $\cA (G|\cV)$. Suppose then
that $X=(V\lra P \lla N)$ is an object of the module category. 
The functor $e$ is a right adjoint to evaluation so there is a map $X\lra
e(V)$ with fibre $X'=(0\lra P'\lla N')$. Since $e(V)$ is a wedge of
copies of $e(\Q)$ it is cellular, so it suffices to show $X'$ is
cellularly equivalent to an object of $\cA (G|\cV)$. Inverting $\cS$
gives a map $(0\lra P'\lla N')\lra (0\lra P'\lla \cSi N')$. The fibre $X''$
is $(0\lra 0\lla \Gamma_{\cS}N')$ with $\Gamma_{\cS}N'$ being
$\cS$-torsion and hence cellular. It remains to observe
$(0\lra P'\lla \cSi N')$ is cellularly trivial, since then we have
$X''\simeq X'$ and a cofibre sequence $X'\lra X \lra e(V). $

To show cellular triviality of $(0\lra P'\lla \cSi N')$ we check maps
out of cells. There are no maps from $\sigma_K$ for $K\in \cK$ because $\cSi N'$ is
torsion free. Finally maps from  $\sigma_{\Gp}$ are calculated from
the pullback square
$$\xymatrix{
  \Hom(\sigma_{\Gp} ,X)\rto \dto &\Hom (\Q , V)\dto \\
\Hom_{\cOcK}(\cOcK, N)\rto &\Hom_{\cOcK}(\cOcK, \cSi \cOcK
\tensor V)    }$$
which in our case is 
$$\xymatrix{
  \Hom (\sigma_{\Gp}, X'')\rto \dto &0\dto\\
  \cSi N'\rto &P'
  }$$

\end{proof}

\section{Towards the general 1-dimensional case}
\label{sec:generalizations}
We consider the restrictions currently imposed, and how they might be
avoided.

\begin{description}
\item [drop the restriction that $\Gp$ is of finite index in its normalizer]

  The first issue is that $E\lr{\Gp}$ needs to be replaced by a
  different ring spectrum. If $\Gp$ is normal in $G$ this means
  $\siftyV{\Gp}\sm DEG/\Gp_+$. Probably in general it should be the
  $\Gp$-localization of the $\Gp$-completion of the sphere.
 Adapting the proof of the pullback square should be a straightforward.
 Under  appropriate group theoretic conditions there is a more explicit
  model of the ring $G$-spectrum, and this will enable calculations of 
homotopy groups to link  to algebra. 
  
  The  effect on the algebraic model is that $\cO_{\Gp}$ will be a more general polynomial 
ring, so the models will be of higher homological dimension. One also needs to 
discuss the relationships between normalizers of $\Gp$ and subgroups in 
$\cK$. If $N_G(\Gp)$ contains all subgroups $N_G(F)$ it is 
straightforward, but this should be weakened, so that it covers cases
like $\Gp=SO(2)$ in $G=SO(3)$. 

\item [drop the restriction that the subgroups $F$ in $\cK$ are 
  finite]

 Our proof of the formality of $DE\lr{\cK}$  involves a certain 
 uniformity of behaviour. It should be sufficient that all elements of $\cK$
 have a common identity component. Since we may always use 
 idempotents, it should suffice that  $\cK$ can be broken into finitely 
 many parts each of which consists of subgroups with a common identity 
 component. 

  \item [drop the restriction that $\cV$ has a single compactifying 
    point] 
    
The general case that $\cV$ is a 1-dimensional spectral space 
introduces significant complication. For example, even with only 
finitely many height 1 subgroups the model needs to incorporate 
the combinatorics of specialization and how it interacts with the 
containments of normalizers. 

\end{description}

 \bibliographystyle{plain}
\bibliography{../../jpcgbib}

\begin{thebibliography}{10}

\bibitem{prismatic}
S.~Balchin, T.~Barthel, and J.~P.~C. Greenlees.
\newblock Prismatic decompositions and rational {$G$}-spectra.
\newblock {\em Preprint, 59pp, arXiv 2311.18808}, 2023.

\bibitem{adelicm}
Scott Balchin and J.~P.~C. Greenlees.
\newblock Adelic models for tensor triangulated categories.
\newblock {\em Advances in Mathematics}, 375:30pp, 2020.

\bibitem{adelic1}
Scott Balchin and J.~P.~C. Greenlees.
\newblock Separated and complete adelic models for one-dimensional {N}oetherian
  tensor-triangulated categories.
\newblock {\em J. Pure Appl. Algebra}, 226(12):Paper No. 107109, 42, 2022.

\bibitem{gtoralq}
David Barnes, J.P.C. Greenlees, and Magdalena K{\c{e}}dziorek.
\newblock An algebraic model for rational toral {$G$}-spectra.
\newblock {\em Algebr. Geom. Topol.}, 19(7):3541--3599, 2019.

\bibitem{EFQ}
J.~P.~C. Greenlees.
\newblock A rational splitting theorem for the universal space for almost free
  actions.
\newblock {\em Bull. London Math. Soc.}, 28(2):183--189, 1996.

\bibitem{ratmack}
J.~P.~C. Greenlees.
\newblock Rational {M}ackey functors for compact {L}ie groups. {I}.
\newblock {\em Proc. London Math. Soc. (3)}, 76(3):549--578, 1998.

\bibitem{o2q}
J.~P.~C. Greenlees.
\newblock Rational {O}(2)-equivariant cohomology theories.
\newblock In {\em Stable and unstable homotopy (Toronto, ON, 1996)}, volume~19
  of {\em Fields Inst. Commun.}, pages 103--110. Amer. Math. Soc., Providence,
  RI, 1998.

\bibitem{s1q}
J.~P.~C. Greenlees.
\newblock Rational {$S\sp 1$}-equivariant stable homotopy theory.
\newblock {\em Mem. Amer. Math. Soc.}, 138(661):xii+289, 1999.

\bibitem{AGconj}
J.~P.~C. Greenlees.
\newblock Triangulated categories of rational equivariant cohomology theories.
\newblock {\em Oberwolfach Reports, pages 480–488, 2006. (cit. on p. 2)},
  2006.

\bibitem{AGtoral}
J.~P.~C. Greenlees.
\newblock Rational equivariant cohomology theories with toral support.
\newblock {\em Algebr. Geom. Topol.}, 16(4):1953--2019, 2016.

\bibitem{bgen}
J.~P.~C. Greenlees.
\newblock Borel cohomology and the relative gorenstein condition for
  classifying spaces of compact lie groups.
\newblock {\em Preprint}, page~10, 2018.

\bibitem{specgq}
J.~P.~C. Greenlees.
\newblock The {B}almer spectrum of rational equivariant cohomology theories.
\newblock {\em J. Pure Appl. Algebra}, 223(7):2845--2871, 2019.

\bibitem{t2wqalg}
J.~P.~C. Greenlees. 
\newblock Spaces of subgroups of toral groups. 
\newblock {\em Preprint, 38pp, arXiv:2501.06914}, 2025. 


\bibitem{t2wq}
J.~P.~C. Greenlees. 
\newblock Rational {$G$}-spectra for small toral groups. 
\newblock {\em Preprint, 31pp}, 2025.

\bibitem{u2q}
J.~P.~C. Greenlees. 
\newblock An algebraic model for rational {$U(2)$}-spectra. 
\newblock {\em In preparation, 15pp}, 2025.

\bibitem{su3q}
J.~P.~C. Greenlees.
\newblock An algebraic model for rational {$SU(3)$}-spectra.
\newblock {\em In preparation, 11pp}, 2025.

\bibitem{gqwf}
J.~P.~C. Greenlees.
\newblock Rational {$G$}-spectra for components with finite {W}eyl groups.
\newblock {\em In preparation, 4pp}, 2025.


\bibitem{AVmodel}
J.~P.~C. Greenlees. 
\newblock An abelian model for rational {$G$}-spectra for a compact {L}ie group 
  {$G$}. 
\newblock {\em In preparation, 27pp}. 

\bibitem{AGnoeth}
J.~P.~C. Greenlees. 
\newblock An algebraic model for rational {$G$}-spectra for finite central 
  extensions of a torus. 
\newblock {\em In preparation, 18pp}. 



\bibitem{cellprin}
J.~P.~C. Greenlees and B.~Shipley.
\newblock The cellularization principle for {Q}uillen adjunctions.
\newblock {\em Homology Homotopy Appl.}, 15(2):173--184, 2013.

\bibitem{gfreeq2}
J.~P.~C. Greenlees and B.~Shipley.
\newblock An algebraic model for free rational {$G$}-spectra.
\newblock {\em Bull. Lond. Math. Soc.}, 46(1):133--142, 2014.

\bibitem{modfps}
J.~P.~C. Greenlees and B.~Shipley.
\newblock Fixed point adjunctions for equivariant module spectra.
\newblock {\em Algebr. Geom. Topol.}, 14(3):1779--1799, 2014.

\bibitem{diagrams}
J.~P.~C. Greenlees and B.~Shipley.
\newblock Homotopy theory of modules over diagrams of rings.
\newblock {\em Proc. Amer. Math. Soc. Ser. B}, 1:89--104, 2014.

\bibitem{tnqcore}
J.~P.~C. Greenlees and B.~Shipley.
\newblock An algebraic model for rational torus-equivariant spectra.
\newblock {\em J. Topol.}, 11(3):666--719, 2018.

\bibitem{PWcomp}
Luca Pol and Jordan Williamson.
\newblock The homotopy theory of complete modules.
\newblock {\em J. Algebra}, 594:74--100, 2022.

\bibitem{ShipleyHZ}
Brooke Shipley.
\newblock {HZ}-algebra spectra are differential graded algebras.
\newblock {\em American Journal of Mathematics}, 129(2):351--379, 2007.

\bibitem{tomDieckTGRT}
Tammo tom Dieck.
\newblock {\em Transformation groups and representation theory}, volume 766 of
  {\em Lecture Notes in Mathematics}.
\newblock Springer, Berlin, 1979.

\end{thebibliography}
\end{document}